\documentclass{article}
\usepackage{bbm}
\usepackage{cancel}

\usepackage{geometry} 

\usepackage{amsfonts, amsmath, amsthm, amssymb, latexsym, graphicx, mathrsfs, tikz, cancel,cleveref}

\newcommand{\intd}{\, \textnormal{d}}
\newcommand{\ud}{\textnormal{d}}

\newcommand{\mb}[1]{\mathbf{#1}}
\newcommand{\mbb}[1]{\mathbb{#1}}

\newcommand{\mcl}[1]{\mathcal{#1}}

\newcommand{\inprod}[2]{\langle{#1},{#2}\rangle}

\newcommand{\supp}{\text{supp}}
\newcommand{\cond}[1]{\, #1 \vert \,}

\theoremstyle{plain}
\newtheorem{theorem}{Theorem}[section]

\newtheorem{lemma}[theorem]{Lemma}

\theoremstyle{definition}

\newtheorem{remark}[theorem]{Remark}

\theoremstyle{remark}

\numberwithin{equation}{section}
\DeclareMathOperator*{\argmin}{arg\,min}

\title{Hausdorff dimension of collision times\\in one-dimensional log-gases}

\author{Nicole Hufnagel, Sergio Andraus${}^1$}

\date{TU Dortmund University, Tsukuba Gakuin University${}^1$\\\today}

\begin{document}
	\maketitle
	
	\abstract{
		We consider systems of multiple Brownian particles in one dimension that repel mutually via a logarithmic potential on the real line, more specifically the Dyson model. These systems are characterized by a parameter that controls the strength of the interaction, $k>0$. In spite of being a one-dimensional system, this system is interesting due to the properties that arise from the long-range interaction between particles. It is a well-known fact that when $k$ is small enough, particle collisions occur almost surely, while when $k$ is large, collisions never occur. However, aside from this fact there was no characterization of the collision times until now. In this paper, we derive the fractal (Hausdorff) dimension of the set of collision times by generalizing techniques introduced by Liu and Xiao to study the return times to the origin of self-similar stochastic processes. In our case, we consider the return times to configurations where at least one collision occurs, which is a condition that defines unbounded sets, as opposed to a single point, namely, the origin. We find that the fractal dimension characterizes the collision behavior of these systems, and establishes a clear delimitation between the colliding and the non-colliding regions in a way similar to that of an order parameter.}
	
	\section{Introduction and main results}
	
	
	Dynamical one-dimensional stochastic systems of charged particles have been an important topic of research in recent decades both in physics and mathematics \cite{Forrester10,Handbook15,LesHouches17}. Perhaps the most well-known of these systems, the Dyson model \cite{Dyson62A}, was formulated as a way to describe the Hamiltonian eigenvalue dynamics of a complex system. Conceptually, it is a system of $N$ particles with equal charge that repel each other via a logarithmic electrostatic potential while performing independent Brownian motions on the real line. It was later rediscovered as a description of long, non-intersecting polymer chains \cite{deGennes68}, and of wetting and melting transitions \cite{Fisher84}. The long-range interaction in this system gives it meaningful physical properties, as it allows for the possibility of a phase transition in contrast to short-range interacting one-dimensional systems \cite[Section 5.6]{Ruelle74}.
	
	The connection with random matrix theory lies in the fact that the eigenvalues of real symmetric, complex Hermitian, and quaternion self-dual matrices with independent Brownian motions as entries up to symmetry obey the stochastic differential equation (SDE) of the Dyson model of parameter $k=1/2,1$, and $2$, respectively. In the random matrix literature, it is commonplace to use the parameter $\beta=2k$, but in this paper we use the notation from Dunkl theory, $k$. The significance of the parameter $k>0$ is that it indicates the strength of the electrostatic interaction relative to the fluctuations due to the Brownian motions that drive the process. Hence, when $k$ is large, fluctuations are small and vice-versa; this allows us to understand this parameter as a proxy for the inverse temperature. The particular case $k=1$ has been thoroughly studied due to its determinantal nature; in this case, the Dyson model can be understood as a system of independent Brownian particles conditioned never to collide \cite{KatoriTanemura07}, it has been extended to non-euclidean geometries \cite{Katori17}, and it has been shown to have a profound link to the multiple Schramm-Loewner Evolution \cite{KatoriKoshida21}.
	
	
	One important feature of the Dyson model is that it can be understood as a multivariate extension of Bessel processes \cite{KaratzasShreveTextbook}, perhaps one of the best-studied stochastic processes after Brownian motion (or Wiener processes). For $t>0$, consider a Brownian motion $(B_t)_{t\geq0}$ and the parameter $k>0$; then, given an initial condition $Y_0=y>0$, the Bessel process $(Y_t)_{t\geq0}$, with $Y_t\geq 0$ for all $t$ is described by the SDE
	\begin{align*}
		\ud Y_t=\ud B_t+k\frac{\ud t}{Y_t}.
	\end{align*}
	We suppress the $k$-dependence on the lhs for simplicity. The transition probability density (TPD) of this process is given by a Bessel function, hence the name. The SDEs that describe the motion of the particles in the Dyson model have this form: the martingale part is a Brownian motion, while the drift term is the sum of the inverse of differences of particle coordinates. The coefficient $k$ plays the same role in both systems, but it also holds the meaning of dimensionality in the Bessel process. Explicitly, if $(B_{t,i})_{t\geq0}$ are independent Brownian motions for $1\leq i\leq d$ and we set $Y_t^{(d)}=\sqrt{\sum_{i=1}^d B_{t,i}^2}$, then the SDE of $(Y_t^{(d)})_{t\geq0}$ coincides with that of $(Y_t)_{t\geq0}$ provided $k=(d-1)/2$ is satisfied. In that case, we call $(Y_t)_{t\geq0}$ a Bessel process of $d$ dimensions. However, the SDE for $(Y_t)_{t\geq0}$ has a unique strong solution for any real positive $k$ \cite{RevuzYor}, so we do not need to restrict ourselves to integer or half-integer values of $k$. Moreover, it is well-known that, given enough time, Bessel processes hit the origin with probability one regardless of the starting point whenever $0<k<1/2$ \cite{KatoriTextbook}. The fact that even in this case Bessel processes have a unique strong solution may be surprising given the form of the SDE. Nevertheless, this situation points to the intuition that for $0<k<1/2$ the repulsion from the origin is not large enough to keep the process from returning to it, which makes this parameter region interesting.
	
	It turns out that the properties of the Bessel process outlined above are shared by the Dyson model: it has a unique strong solution for all $k>0$ \cite{CepaLepingle, GraczykMalecki}, and for $0<k<1/2$, two particles will collide almost surely given enough time \cite{Demni09}. Another important aspect is that the Dyson model can be formulated as the continuous (or ``radial'') part of a Dunkl process \cite{Dunkl89, RoslerVoit98, GallardoYor05} of type $A_{N-1}$. Each type of radial Dunkl process obeys a set of SDEs similar to that of a Bessel process. Because of this similarity, radial Dunkl processes are also called \emph{multivariate Bessel processes}; for the remainder of the paper we will refer to them as such, that is, we will refer to the Dyson model as the multivariate Bessel process of type $A_{N-1}$ \cite{DunklXu}.
	
	The main advantage of the Dunkl process formulation of these systems is that, in contrast with the random matrix formulation, the parameter $k$ is never required to take any particular value. It is then natural to consider their behavior, for example, when $k$ tends to infinity, a topic that has been thoroughly studied \cite{Fekete23, Deift, AKM12, AKM14, AV19A, AV19B, AHV, VW}. Another natural case to consider is the main focus of this paper, the case where $0<k<1/2$, and collisions occur almost surely. Given that collisions characterize the behavior of these systems in this case, it is natural to wonder about the properties of visiting times to the origin of the single-variable Bessel process. There is a wealth of research on the subject \cite{Kent78, PitmanYor82, PitmanYor99A, PitmanYor99B}, but the most relevant for our purposes is the work of Liu and Xiao \cite{LiuXiao98}. In their paper, they introduce a strategy to derive the Hausdorff dimension, a particular definition of fractal dimension, of visiting times to the origin of stochastic processes that are \emph{semi-stable}, or in other words, self-similar. For example, a Brownian motion $(B_t)_{t\geq0}$ started at the origin is $1/2$-semi-stable because it satisfies the equality
	\[B_{ct}=c^{1/2}B_{t}\]
	in distribution for every $c>0$. It is a straightforward task to show that the Bessel process $(Y_t)_{t\geq 0}$ satisfies all of the conditions that lead to Theorem~4.1 in \cite{LiuXiao98}. Namely, $(Y_t)_{t\geq0}$ is $1/2$-semi-stable and its transition density obeys a particular power-law asymptotic form near the origin (see Lemma~\ref{lm:boundsA} below), so the Hausdorff dimension of the visiting times to the origin for $(Y_t)_{t\geq0}$ is
	\[\frac{1}{2}\Big(1-\min(1,2k)\Big).\]
	
	In view of this property of Bessel processes, one may ask whether one can make a similar statement for multivariate Bessel processes, as they can be shown to be $1/2$-semi-stable. In fact, the notion of a phase transition provides a physical motivation for this question; thermodynamic functions for a multiple-particle system are believed to be piecewise-analytically dependent on the parameters of the system ($k$ in this case) and singularities indicate the existence of a phase transition. In our case we do not consider a thermodynamic function, but such a piecewise continuous nature of the Hausdorff dimension on $k$ for multivariate Bessel processes would indicate a behavior similar to a phase transition at $k=1/2$. The objective of this paper is to answer this question in the affirmative for the case $A_{N-1}$, namely, the Dyson model. Henceforth, all objects defined here refer to the root system of type $A_{N-1}$. Let $(X_t)_{t\geq 0}$ denote these processes with initial configuration $X_0=x$, and let $W$ denote the open subset of $\mbb{R}^N$ where they are defined, and let $\overline{W}$ be its closure. Moreover, denote by $\partial W$ the boundary of $W$, namely, the configurations where a collision occurs. The main result is summarized in the following theorem.
	\begin{theorem}\label{th:main}
		Fix $x\in \overline W$, and define the set of times the process $X_t$ spends at the boundary of the Weyl chamber $\partial W$ by $X^{-1}(\partial W)$, that is,
		\[X^{-1}(\partial W):=\{t\geq 0\big| X_t\in\partial W\}.\] 
		The Hausdorff dimension of collision times for the multivariate Bessel processes of type $A_{N-1}$ is given by
		\begin{equation}
			\dim\big( X^{-1}(\partial W)\big)=\frac{1}{2}\Big(1-\min(1,2k)\Big)\notag
		\end{equation}
		$\mbb P^x-$almost surely.
	\end{theorem}
	Explicit definitions of the objects above will be given in the following section. The main difficulty in obtaining this result lies in the type of event we are interested in observing. In \cite{LiuXiao98}, the event is the return to the origin, whereas in our case the event is comprised of all configurations such that $X_{t,i}=X_{t,j}$ for some $i\neq j$. This means that in order to extend Liu and Xiao's machinery to our setting, we must derive similar asymptotics for the transition function in a set that is not bounded. Our key contribution is the derivation of a new asymptotic behavior of the probability of a collision being close to occur, which is a condition that defines an unbounded set and which is fundamentally different from the case covered in \cite{LiuXiao98}. The key to overcome this hurdle lies in making use of an asymptotic relation by Graczyk and Sawyer which removes the difficulty of working with Dunkl-Bessel functions \cite{GraczykSawyer}, decomposing the collision events into simpler events, and performing adequate coordinate transformations to obtain the corresponding asymptotics that lead to the theorem.
	
	The physical significance of this result is that the Hausdorff dimension characterizes the behavior of these processes, highlighting the fact that for $k\geq 1/2$ no collisions occur, while when $k<1/2$ collisions are frequent enough that the set of collision times has fractal structure. In this sense, with $k$ as main parameter one can identify an ordered (non-colliding) and a disordered (colliding) parameter region, as the Hausdorff dimension behaves very much like a thermodynamic function in this case. While there is precedent of an abstract phase transition in these systems \cite{ValkoVirag}, the results obtained in the present work point to a phenomenon based on collisions, which are simple, concrete events. Another aspect that must be pointed out is that the matrix eigenvalue formulations of these processes show no collisions if they are driven by Brownian motions; this is not the case for \emph{fractional} Brownian motions, which can overcome the log-potential repulsion between eigenvalues \cite{SongXiaoYuan21}.
	
	We also point out that Theorem~\ref{th:main} seems to indicate a Hausdorff dimension of $1/2$ for the limit $k=0$. Our result does not cover this case, but for independent Brownian motions one can readily apply the machinery of \cite{LiuXiao98} to the difference of two Brownian motions, itself a one-dimensional Brownian motion, to see that the collision times have a Hausdorff dimension of $1/2$, as expected.
	
	We present the proof of our main statement in two parts, an upper bound part and a lower bound part. For the upper bound, we impose a small technical condition on the initial configuration of the multivariate Bessel processes in order to control the integrals involved in the derivation of asymptotics near collisions, and we then follow the framework of Liu and Xiao. 
	For the lower bound, we rely on the capacity argument, otherwise known as Frostman's theorem \cite{Kahane94}. However, here we must prove a new lemma to eliminate the technical condition mentioned above. Moreover, we prove the lower bound of the Hausdorff dimension over intervals starting from a small but positive initial time, $\epsilon$, as a modification of the Liu-Xiao proof. After this, we follow their strategy to extend the proof to intervals starting from zero. 
	
	The paper is structured as follows. We give the definition of the multivariate Bessel processes of type $A_{N-1}$, as well as that of all related quantities in Section~\ref{sc:setting}. We also give a brief review of Hausdorff dimensions and some related lemmas. In Section~\ref{sc:upper}, we present the proof of the upper bound of the Hausdorff dimension of collision times. We prove several lemmas related to asymptotics of the transition functions near collision events, and we follow them with the proof of the main statement of the section. We follow the same basic structure in Section~\ref{sc:lower} to prove the lower bound, which completes the proof of Theorem~\ref{th:main}. We finish the paper by making comments on these results and by discussing open problems in Section~\ref{sc:conclusions}.
	
	\section{Setting, definitions, and properties}\label{sc:setting}
	\subsection{Multivariate Bessel processes}
	In this section, we briefly introduce the diffusions and notations used throughout the paper. For more details we refer to \cite{DunklXu,DunklProcesses}.  We denote with $(\Omega,\mcl{F},\mbb{F},\mbb{P})$ the underlying probability space, where $\mbb{F}:=(\mathcal{F}_t)_{t\geq 0}$ is a filtration with the natural $\sigma$-algebra $\mathcal{F}_t:=\sigma(X_s : s \leq t)$ corresponding to the respective process. For a neater notation we use $\mbb{P}^x(A):=\mbb{P}(A\cond{} X_0=x)$ with $A\in \Omega$. 
	
	We define by
	\begin{align*}
		W= W_{A_{N-1}}:= \{x\in \mathbb{R}^N\cond{} x_1<x_2<\dots<x_N\}.
	\end{align*}
	the open Weyl chamber of type $A_{N-1}$. 
	
	The multivariate Bessel process of type $A_{N-1}$ is the solution of the stochastic differential equation 
	\begin{align*}
		\ud X_{t,i}=\ud B_{t,i}+k\sum\limits_{\substack{j=1\\j\neq i}}^N\frac{\ud t}{X_{t,i}-X_{t,j}},\\
	\end{align*}
	for $i=1, \dots, N$ with initial condition $X_0=x\in \overline{W}$, where $(B_t)_{t\geq 0}$ is a standard multivariate Brownian motion, that is, each of its components is an independent Brownian motion.
	Also, $X_t\in \overline{W}$ for all $t$ almost surely, and the solution to this SDE can be continued after collisions with $\partial W$ \cite{CepaLepingle}. It is known that this set of SDEs has a unique strong solution for $k>0$ \cite{CepaLepingle,GraczykMalecki,Demni09}. From a physical point of view, the dimension $N$ represents the number of particles, while $k> 0$, representing the inverse temperature, regulates the interactions. These are Feller processes, and as such they satisfy the strong Markov property; this fact follows from \cite[Proposition 4.5]{RoslerVoit98} and \cite[Theorem 4.8]{Rosler98} after specialization to the root system of type $A_{N-1}$.
	
	Furthermore, we consider the edge set
	\begin{align*}
		E^r:=\left\{x\in \overline W\cond{} \exists\ i\in \{1,\dots, N-1\} : 0\leq x_{i+1}-x_i\leq r\right\}
	\end{align*}
	with thickness $r\geq 0$ which is an auxiliary set for later calculations of the hitting times with $\partial W=E^0$. In particular, $cE^r=E^{cr}$ for $r,c>0$ is an useful property which we will often use when performing transformations.
	
	In the case $0<k<\frac{1}{2}$ the corresponding process almost surely hits the edge $\partial W$ \cite{Demni09}. We denote the standard inner product of $x,y\in\mbb{R}^N$ by $\inprod{x}{y}$, and the $2$-norm by $\|x\|=\sqrt{\inprod{x}{x}}$. The transition probability density of $(X_t)_{t\geq 0}$ ending at $y$ in the corresponding Weyl chamber after a time $t>0$ is expressed as \begin{align*}
		p_k(t,y\cond{}x):= \frac{e^{-\frac{\|x\|^2+\|y\|^2}{2t}}}{c_kt^{\frac{N}{2}}}D_k\left(\frac{x}{\sqrt{t}},\frac{y}{\sqrt{t}}\right)w_k
		\left(\frac{y}{\sqrt{t}}\right).
	\end{align*}
	The weight function $w_k$ is specified by
	\begin{align*}
		w_k(x):=\prod\limits_{1\leq i<j\leq N}(x_j-x_i)^{2k},
	\end{align*}
	$D_k$ is the Dunkl-Bessel function \cite{BakerForrester97}, which satisfies the bounds
	\[e^{-\|x\|\|y\|}\leq D_k(x,y)\leq e^{\|x\|\|y\|}.\]
	The normalization constant $c_k$ depends both on $k$ and $N$; we will leave this dependence implicit. We also introduce the notation
	\begin{align*}
		P_k(t,A\cond{} x)=\int\limits_A p_k(t,y\cond{}x)\intd y
	\end{align*}
	for $A\subseteq \overline{W}$. For later calculations we introduce the sum of multiplicities
	\begin{align}
		\gamma:=\sum\limits_{1\leq i<j\leq N}k=\frac{kN(N-1)}{2}.\label{eq:gamma}
	\end{align}
	
	\subsection{Hausdorff dimension - Definition and properties}
	The Hausdorff dimension is a generalization of the well-known dimension concept. This means, well-known geometric objects like straight lines, hyperplanes and others with intuitive dimensionality keep the same dimension. The Hausdorff dimension offers a finer distinction, since it admits not only natural numbers. 
	
	We denote by $B(x,R):=\{y\in \mbb{R}^N:\|y\|\leq R\}$ the closed $N$-dimensional ball centered at $x$ with radius $R$. For the monotonically increasing monomial (on the positive halfline) of power $\alpha\geq 0 $ the Hausdorff measure of $E\subset \mathbb{R}^N$ is defined as 
	\begin{align*}
		m_\alpha(E):=\lim\limits_{\varepsilon\to 0} \inf\Big\{\sum_{i=1}^\infty (2r_i)^\alpha: \exists r_i<\varepsilon \text{ with } E\subset \bigcup\limits_{i=1}^\infty B(x_i,r_i)\Big\}.
	\end{align*}
	The radius $r_i$ may be equal to zero, and therefore the covering can be a finite union. The Hausdorff dimension is defined by the following lemma. 
	\begin{lemma}[{\cite[8.1 Hausdorff dimension]{Adler1981}}]
		\label{lm_Hausdorff_dimension}
		For any set $E\subset \mathbb{R}^n$ there exists a unique number $\alpha^\star$, called the Hausdorff dimension of $E$, for which 
		\begin{align*}
			\alpha <\alpha^\star \Rightarrow m_\alpha(E)=\infty, \quad \alpha >\alpha^\star \Rightarrow m_\alpha(E)=0. 
		\end{align*}
		This number is denoted by $\dim (E)$ and satisfies
		\begin{align*}
			\alpha^\star=\dim(E)=\sup\{\alpha>0:m_\alpha(E)=\infty \}=\inf \{\alpha>0:m_\alpha(E)=0\}.
		\end{align*}
	\end{lemma}
	This lemma implies that $m_\alpha(E)\in\{0,\infty\}$, so if it is finite it must be zero. The Hausdorff dimension satisfies the following properties.
	\begin{itemize}
		\item \textit{Countable stability \cite[2.2 Hausdorff dimension]{Falconer2013fractal}} If $F_1,F_2,\dots $ is a (countable) sequence of sets then $$\dim \Big(\bigcup\limits_{i=1}^\infty F_i\Big)=\sup\limits_{i\in\mathbb{N}}\dim (F_i).$$
		\item\textit{Monotonicity \cite[2.2 Hausdorff dimension]{Falconer2013fractal}} If $E\subset F$ then $\dim (E)\leq \dim (F)$.
	\end{itemize}
	
	An often used tool is the capacity argument. 
	\begin{lemma}[{\cite[Lemma 8.2.3]{Adler1981}}]
		\label{lem_capacity}
		Let $A$ be a compact subset of $\mathbb{R}^n$. Suppose there exists a positive measure $\mu$ concentrated on $A$, $\mu(\mathbb{R}^n)=\mu(A)$, 
		and a finite $\alpha $ such that the energy integrals 
		\begin{align*}
			I_\beta(\mu)=\int\limits_A\int\limits_A \frac{\ud \mu (x) \ud \mu(y)}{\|x-y\|^\beta} 
		\end{align*}
		are finite for all $\beta<\alpha$. Then $\dim(A)\geq \alpha.$
	\end{lemma}
	Another version of this statement is given in \cite[p. 133]{Kahane94}. There, it is established that the supremum value of $\beta$ such that there exists a measure $\mu$ for which $I_\beta(\mu)$ is finite, is nothing but the Hausdorff dimension of $A$.
	
	\section{Upper bound}\label{sc:upper}
	We perform the proof of \Cref{th:main} for $x\in \partial W$ and afterwards we can easily conclude the results for any $x\in \overline{W}$ thanks to the finiteness of the first hitting time and the countable stability of the Hausdorff dimension. The upper bound of the collision times' Hausdorff dimension is given by the following statement.
	\begin{theorem}
		\label{theo_multidimensional_upper_bound}
		For every $0<k<\frac{1}{2}$ and $x\in  \partial{W}
		$ the inequality
		\begin{align*}
			\dim \big(X^{-1}(\partial W
			)\big)\leq \frac{1}{2}-k
		\end{align*}
		is $\mbb{P}^x$-almost surely valid. 
	\end{theorem}
	We prove this statement in a series of lemmas. The  cornerstone of our proof is the probability distribution's asymptotic behavior when a collision is close to occurring, that is, when two particles are closer than a distance $r>0$ to each other. The rest of the proof is mostly analogous to the machinery by Liu and Xiao \cite{LiuXiao98} except for a few technical details.
	\begin{lemma}\label{lm:boundsA}
		We fix $k>0$, $R>0$ and $x\in E^r\cap B(0,R)$. Given $r>0$, there exist constants $C_1(k,N)>0$ and $C_2(k,N,R)>0$ such that
		\begin{equation*}
			C_1\min\{1, r^{2k+1}\}\leq P_k(1,E^r
			\cond{}x)\leq C_2\min\{ 1,r^{2k+1}\}.
		\end{equation*}
	\end{lemma}
	\begin{remark}
		The main difference to the proof for the classical Bessel process case is that $C_2$ depends on $R$.
	\end{remark}
	\begin{proof}[Proof of Lemma~\ref{lm:boundsA}]
		We will make use of the asymptotic relationship by Graczyk and Sawyer \cite[Thm 3.1]{GraczykSawyer} for this proof; there exist constants $0< K_1(N,k)\leq K_2(N,k)$ such that
		\begin{align}
			\frac{K_1t^{-N/2}e^{-\|y-x\|^2/(2t)}w_k(y)}{\prod_{1\leq l<m\leq N}(t/2+(x_m-x_l)(y_m-y_l))^k} \leq p_k(t,y\cond{}x)\leq \frac{K_2t^{-N/2}e^{-\|y-x\|^2/(2t)}w_k(y)}{\prod_{1\leq l<m\leq N}(t/2+(x_m-x_l)(y_m-y_l))^k}.\label{eq:GraczykSawyer}
		\end{align}
		We begin with the lower bound. We choose $i$ as follows,
		\begin{align*}
			i=\argmin_{j=1,\ldots,N-1}(x_{j+1}-x_j)
		\end{align*}
		and define the set
		\begin{align}
			E_i^r:=\{y\in E^r:0\leq y_{i+1}-y_i\leq r
			\}\subset E^r.\label{eq:SetEi}
		\end{align}
		Note that this $i$ may not be unique; in that case, we may choose one of the possible candidates arbitrarily, say, the smallest possible such $i$. By \eqref{eq:GraczykSawyer}, 
		\begin{align*}
			P_k(1,E^r\cond{}x)&\geq K_1\int\limits_{E^r}\frac{e^{-\|y-x\|^2/2}w_k(y)}{\prod_{1\leq l<m\leq N}(1/2+(x_m-x_l)(y_m-y_l))^k}\intd y\\
			&\geq K_1\int\limits_{E_i^r}\frac{e^{-\|y-x\|^2/2}w_k(y)}{\prod_{1\leq l<m\leq N}(1/2+(x_m-x_l)(y_m-y_l))^k}\intd y.
		\end{align*}
		Now, we consider the case $0<r<1$ and perform the following variable substitution,
		\begin{align}
			y_j&=x_j+z_j,\quad j\neq i,i+1\notag\\
			y_i&=\frac{z_i-z_{i+1}}{\sqrt{2}}+\frac{x_i+x_{i+1}}{2},\label{eq:maintransformation}\\
			y_{i+1}&=\frac{z_i+z_{i+1}}{\sqrt{2}}+\frac{x_i+x_{i+1}}{2},\notag
		\end{align}
		with its inverse given by
		\begin{align*}
			z_j&=y_j-x_j,\quad j\neq i,i+1\\
			z_i&=\frac{y_i+y_{i+1}}{\sqrt{2}}-\frac{x_i+x_{i+1}}{\sqrt2},\\
			z_{i+1}&=\frac{y_{i+1}-y_i}{\sqrt{2}}\geq 0.
		\end{align*}
		These equations mean that there exists a linear operator $A$ and a fixed vector $b$ such that $y=Az + b$ and $|\det A|=1$. From these equations, it is clear that $z_{i+1}$ represents the distance between the $i$th and $(i+1)$th particles. As we consider how $Az+b\in E_i^r$ imposes restrictions on $z$, we will eliminate the dependence on $x$ by making these restrictions stronger. For $j\neq i-1,i,i+1$ we obtain 
		\begin{align*}
			z_j+x_j\leq z_{j+1}+x_{j+1},
		\end{align*}so 
		\begin{align*}
			z_j\leq z_{j+1}+x_{j+1}-x_j.
		\end{align*}
		Because $x_{j+1}-x_j\geq 0$, a stronger condition is given by
		\begin{align*}
			z_j\leq z_{j+1}.
		\end{align*}
		Because $0\leq y_{i+1}-y_i\leq r\leq 1$, we obtain for $z_{i+1}$
		\begin{align*}
			0\leq z_{i+1}\leq r/\sqrt2\leq 1/\sqrt2.
		\end{align*}
		Next, $z_{i-1}$ satisfies the bound
		\begin{align*}
			z_{i-1}+x_{i-1}\leq \frac{z_i-z_{i+1}}{\sqrt{2}}+\frac{x_i+x_{i+1}}{2},
		\end{align*}
		or
		\begin{align*}
			z_{i-1}\leq \frac{z_i-z_{i+1}}{\sqrt{2}}+\frac{(x_i-x_{i-1})+(x_{i+1}-x_{i-1})}{2},
		\end{align*}
		so using the fact that the $x_i$'s are ordered and that $z_{i+1}\leq 1/\sqrt2$, we impose the stronger condition 
		\begin{align*}
			z_{i-1}\leq \frac{z_i}{\sqrt{2}}-\frac12.
		\end{align*}
		We obtain the upper bound for $z_i$ by using $y_{i+1}\leq y_{i+2}$. This condition reads
		\begin{align*}
			\frac{z_i+z_{i+1}}{\sqrt{2}}+\frac{x_i+x_{i+1}}{2} \leq z_{i+2}+x_{i+2}\ \Leftrightarrow\  \frac{z_{i}}{\sqrt2}+\frac{z_{i+1}}{\sqrt{2}}\leq  z_{i+2}+\frac{(x_{i+2}-x_i)+(x_{i+2}-x_{i+1})}{2},
		\end{align*}
		so we impose the stronger condition 
		\begin{align*}
			\frac{z_{i}}{\sqrt2}+\frac12\leq z_{i+2},
		\end{align*}
		which translates into $z_i\leq \sqrt2 z_{i+2}-1/\sqrt2$ for $z_i$. Because all of these conditions are stronger than those imposed on $y$, it follows that for every $z$ that satisfies
		\begin{align}
			\begin{split}
				z_j&\leq z_{j+1},\ j\neq i-1,i,i+1,\\
				z_{i-1}&\leq \frac{z_i}{\sqrt2}-\frac12,\\
				\frac{z_i}{\sqrt2}+\frac12&\leq z_{i+2},\\
				0&\leq z_{i+1}\leq r/\sqrt2,
			\end{split}
			\label{eq:boundsLambda}
		\end{align}
		$Az+b$ must be an element of $E_i^r$. 
		If we further restrict $z$ to satisfy
		\begin{align*}
			z&\in \Lambda_{i,N}\cap\{0\leq z_{i+1}\leq r/\sqrt2\},\\
			\Lambda_{i,N}:=\Big\{z\in\mathbb{R}^N:&\sum_{\substack{j=1\\j\neq i+1}}^Nz_j^2\leq (N+1)^3,\ z_j\leq z_{j+1}\ (j\neq i-1,i,i+1),
			z_{i-1}+\frac12\leq  \frac{z_i}{\sqrt{2}}\leq z_{i+2}-\frac12\Big\},
		\end{align*}
		we see that $z$ belongs to a compact set which guarantees that $Az+b\in E_i^r$. 
		Obviously, we can see  
		\begin{align*}
			\underbrace{\Big[0,\frac12\Big]\times \dots \times \Big[\frac{i-2}{2},\frac{i-1}{2}\Big]}_{\ni(z_1,\dots,z_{i-1})}&\times \Big[\frac{i}{\sqrt2 },\frac{i+1}{\sqrt2}\Big]\times \Big[0,\frac{r}{\sqrt2}\Big]
			\times \Big[\frac{i+2}{2},i+2\Big]\\
			&\times\underbrace{\Big[i+2,i+3\Big]\times \dots \times \Big[N-1,N\Big]}_{\ni (z_{i+3},\dots, z_N)}\subset \Lambda_{i,N}
		\end{align*}
		since every vector within the set on the left-hand side fulfills
		\begin{align*}
			\sum_{\substack{j=1\\j\neq i+1}}^Nz_j^2&\leq \sum_{j=1}^{i-1}\frac{j^2}{4} + \frac{(i+1)^2}{2}+\sum_{j=i+2}^N j^2\\
			&\leq \sum_{j=1}^N j^2= \frac{N(N+1)(2N+1)}{6}\\
			&\leq (N+1)^3.
		\end{align*}
		We define the projection
		\begin{align*}
			p_{i+1}(\xi_1,\ldots,\xi_i,\xi_{i+1},\xi_{i+2},\ldots,\xi_N):=(\xi_1,\ldots,\xi_i,\xi_{i+2},\ldots,\xi_N)
		\end{align*}
		and observe that the Lebesgue measure of the projection of $\Lambda_{i,N}$ is bounded below by
		\begin{align*}
			\lambda^{N-1}[p_{i+1}(\Lambda_{i,N})]\geq \frac{1}{2^N}.
		\end{align*}
		Now, for every such $z$ we can derive $x$-independent lower bounds for the integrand as follows. First, the argument in the exponential can be bounded by
		\begin{align*}
			\|x-y\|^2&=\sum_{\substack{j=1\\j\neq i,i+1}}^N(y_j-x_j)^2+(y_i-x_i)^2+(y_{i+1}-x_{i+1})^2\\
			&=\sum_{\substack{j=1\\j\neq i,i+1}}^Nz_j^2+\Big(\frac{z_i-z_{i+1}}{\sqrt{2}}+\frac{x_i+x_{i+1}}{2}-x_i\Big)^2+\Big(\frac{z_i+z_{i+1}}{\sqrt{2}}+\frac{x_i+x_{i+1}}{2}-x_{i+1}\Big)^2\\
			&=\sum_{\substack{j=1\\j\neq i,i+1}}^Nz_j^2+\Big(\frac{z_i-z_{i+1}}{\sqrt{2}}+\frac{x_{i+1}-x_i}{2}\Big)^2+\Big(\frac{z_i+z_{i+1}}{\sqrt{2}}-\frac{x_{i+1}-x_i}{2}\Big)^2\\
			&=\sum_{\substack{j=1\\j\neq i+1}}^Nz_j^2+z_{i+1}^2+\frac{(x_{i+1}-x_i)^2}{2}+(x_{i+1}-x_i)\frac{z_i-z_{i+1}}{\sqrt{2}}-(x_{i+1}-x_i)\frac{z_i+z_{i+1}}{\sqrt{2}}\\
			&=\sum_{\substack{j=1\\j\neq i+1}}^Nz_j^2+z_{i+1}^2+\frac{(x_{i+1}-x_i)^2}{2}-\sqrt2 (x_{i+1}-x_i)z_{i+1}\leq (N+1)^3+z_{i+1}^2+\frac{(x_{i+1}-x_i)^2}{2}\\
			&\leq (N+1)^3 + \frac{r^2}{2}+ \frac{(x_{i+1}-x_i)^2}{2}\leq (N+1)^3 + \frac{r^2}{2}+ \frac{r^2}{2}\leq (N+1)^3+1.
		\end{align*}
		Next, we must find bounds for the terms in the product
		\begin{align*}
			\prod_{1\leq l<m\leq N}\frac{(y_m-y_l)^2}{1/2+(x_m-x_l)(y_m-y_l)}
		\end{align*}
		to the power $k$. Now, we consider six separate cases. The simplest case is that where $l=i$, $m=i+1$,
		\begin{align*}
			\frac{(y_{i+1}-y_i)^2}{1/2+(x_{i+1}-x_i)(y_{i+1}-y_i)}=\frac{(\sqrt2 z_{i+1})^2}{1/2+\sqrt2 z_{i+1}(x_{i+1}-x_i)}\geq \frac43 z_{i+1}^2.
		\end{align*}
		Here, we have used the fact that both $\sqrt2 z_{i+1}$ and $x_{i+1}-x_i$ are less than or equal to $r$, which in turn is bounded above by 1. Next is the case $l,m\neq i,i+1$, where we get
		\begin{align*}
			\frac{(y_m-y_l)^2}{1/2+(y_m-y_l)(x_m-x_l)}=\frac{(z_m-z_l+x_m-x_l)^2}{1/2+(z_m-z_l+x_m-x_l)(x_m-x_l)}\geq \frac{(z_m-z_l+x_m-x_l)^2}{1/2+(z_m-z_l+x_m-x_l)^2}.
		\end{align*}
		The last member uses the fact that $z\in \Lambda_{i,N}$ and therefore $z_m-z_l\geq 0$. Next, observing that the function $\xi\mapsto\xi^2/(1/2+\xi^2)$ is increasing in $\xi\geq0$, we conclude that
		\begin{align*}
			\frac{(y_m-y_l)^2}{1/2+(y_m-y_l)(x_m-x_l)}\geq \frac{(z_m-z_l)^2}{1/2+(z_m-z_l)^2}
		\end{align*}
		because $x\in \overline{W}$ implies that $x_m-x_l\geq 0$. The following case is $l=i$, $m>i+1$, where we have
		\begin{align*}
			\frac{(y_m-y_i)^2}{1/2+(y_m-y_i)(x_m-x_i)}&=\frac{(z_m-(z_i-z_{i+1})/\sqrt2 +x_m-(x_i+x_{i+1})/2)^2}{1/2+(z_m-(z_i-z_{i+1})/\sqrt2+x_m-(x_i+x_{i+1})/2)(x_m-x_i)}\\
			&=\frac{(z_m-(z_i-z_{i+1})/\sqrt2 -(x_{i+1}-x_i)/2+x_m-x_i)^2}{1/2+(z_m-(z_i-z_{i+1})/\sqrt2 -(x_{i+1}-x_i)/2+x_m-x_i)(x_m-x_i)}\\
			&\geq\frac{(z_m-(z_i-z_{i+1})/\sqrt2 -(x_{i+1}-x_i)/2+x_m-x_i)^2}{1/2+(z_m-(z_i-z_{i+1})/\sqrt2 -(x_{i+1}-x_i)/2+x_m-x_i)^2}\\
			&\geq\frac{(z_m-z_i/\sqrt2-1/2)^2}{1/2+(z_m-z_i/\sqrt2 -1/2)^2}.
		\end{align*}
		Here, we have used the fact that $2z_m-\sqrt2(z_i-z_{i+1})-(x_{i+1}-x_i)\geq 0$ because of $z_m\geq z_{i+2}\geq z_i/\sqrt2 +1/2$ and $x_{i+1}-x_i\leq 1$, as well as the non-negativity of $z_{i+1}$ and the increasing nature of the function $\xi\mapsto\xi^2/(1/2+\xi^2)$ for $\xi\geq 0$. Similarly, when $l=i+1$ and $m>i+1$ we have
		\begin{align*}
			\frac{(y_m-y_{i+1})^2}{1/2+(y_m-y_{i+1})(x_m-x_{i+1})}&=\frac{(z_m-(z_i+z_{i+1})/\sqrt2 +x_m-(x_i+x_{i+1})/2)^2}{1/2+(z_m-(z_i+z_{i+1})/\sqrt2+x_m-(x_i+x_{i+1})/2)(x_m-x_{i+1})}\\
			&=\frac{(z_m-(z_i+z_{i+1})/\sqrt2 +(x_{i+1}-x_i)/2+x_m-x_{i+1})^2}{1/2+(z_m-(z_i+z_{i+1})/\sqrt2 +(x_{i+1}-x_i)/2+x_m-x_{i+1})(x_m-x_{i+1})}\\
			&\geq\frac{(z_m-(z_i+z_{i+1})/\sqrt2 +(x_{i+1}-x_i)/2+x_m-x_{i+1})^2}{1/2+(z_m-(z_i+z_{i+1})/\sqrt2 +(x_{i+1}-x_i)/2+x_m-x_{i+1})^2}\\
			&\geq\frac{(z_m-z_i/\sqrt2-1/2)^2}{1/2+(z_m-z_i/\sqrt2-1/2)^2}.
		\end{align*}
		The last two cases are those for which $l<i$ and $m=i$ or $i+1$. For the first one we obtain
		\begin{align*}
			\frac{(y_i-y_l)^2}{1/2+(y_i-y_l)(x_i-x_l)}&=\frac{((z_i-z_{i+1})/\sqrt2-z_l+(x_i+x_{i+1})/2-x_l)^2}{1/2+((z_i-z_{i+1})/\sqrt2-z_l+(x_i+x_{i+1})/2-x_l)(x_i-x_l)}\\
			&=\frac{((z_i-z_{i+1})/\sqrt2-z_l+(x_{i+1}-x_i)/2+x_i-x_l)^2}{1/2+((z_i-z_{i+1})/\sqrt2-z_l+(x_{i+1}-x_i)/2+x_i-x_l)(x_i-x_l)}\\
			&\geq\frac{((z_i-z_{i+1})/\sqrt2-z_l+(x_{i+1}-x_i)/2+x_i-x_l)^2}{1/2+((z_i-z_{i+1})/\sqrt2-z_l+(x_{i+1}-x_i)/2+x_i-x_l)^2}\\
			&\geq\frac{(z_i/\sqrt2-z_l-1/2)^2}{1/2+(z_i/\sqrt2-z_l-1/2)^2}.
		\end{align*}
		For this we used a similar argument to those used for the previous terms, as well as the inequality $z_l\leq z_{i-1}\leq z_i/\sqrt2-1/2$. In the same way, we get
		\begin{align*}
			\frac{(y_{i+1}-y_l)^2}{1/2+(y_{i+1}-y_l)(x_{i+1}-x_l)}&=\frac{((z_i+z_{i+1})/\sqrt2-z_l+(x_i+x_{i+1})/2-x_l)^2}{1/2+((z_i+z_{i+1})/\sqrt2-z_l+(x_i+x_{i+1})/2-x_l)(x_{i+1}-x_l)}\\
			&=\frac{((z_i+z_{i+1})/\sqrt2-z_l-(x_{i+1}-x_i)/2+x_{i+1}-x_l)^2}{1/2+((z_i+z_{i+1})/\sqrt2-z_l-(x_{i+1}-x_i)/2+x_{i+1}-x_l)(x_{i+1}-x_l)}\\
			&\geq\frac{((z_i+z_{i+1})/\sqrt2-z_l-(x_{i+1}-x_i)/2+x_{i+1}-x_l)^2}{1/2+((z_i+z_{i+1})/\sqrt2-z_l-(x_{i+1}-x_i)/2+x_{i+1}-x_l)^2}\\
			&\geq\frac{(z_i/\sqrt2-z_l-1/2)^2}{1/2+(z_i/\sqrt2-z_l-1/2)^2}
		\end{align*}
		for $m=i+1$ and $l<i$.
		For a more compact notation, we define the vector
		\begin{align}
			z^*_{i+1}:=p_{i+1}(z)=(z_1\ldots,z_i,z_{i+2},\ldots,z_N).\label{eq:projection}
		\end{align}
		Using all of the previous bounds, we can write
		\begin{align*}
			\mathbb{P}^x&(X_1\in E^r)\\
			&\geq K_1(N,k)\int\limits_{p_{i+1}(\Lambda_{i,N})}e^{-((N+1)^3+1)/2}\Pi_{N,k,i}^{(1)}(z^*_{i+1})\Pi_{N,k,i}^{(2)}(z^*_{i+1})\Pi_{N,k,i}^{(3)}(z^*_{i+1})\intd z^*_{i+1}\int\limits_0^{r/\sqrt2}\frac{4^k }{3^k}z_{i+1}^{2k}\intd z_{i+1}\\
			&=K_1(N,k)e^{-((N+1)^3+1)/2}\frac{2^k}{3^k\sqrt2}r^{2k+1}\int\limits_{p_{i+1}(\Lambda_{i,N})}\Pi_{N,k,i}^{(1)}(z^*_{i+1})\Pi_{N,k,i}^{(2)}(z^*_{i+1})\Pi_{N,k,i}^{(3)}(z^*_{i+1})\intd z^*_{i+1},
		\end{align*}
		where the products $\Pi_{N,k,i}^{(1)}(z^*_{i+1})$, $\Pi_{N,k,i}^{(2)}(z^*_{i+1})$, and $\Pi_{N,k,i}^{(3)}(z^*_{i+1})$ are given by
		\begin{align*}
			\Pi_{N,k,i}^{(1)}(z^*_{i+1})&:=\prod_{\substack{1\leq l<m\leq N\\l,m\neq i,i+1}}\left[\frac{(z_m-z_l)^2}{1/2+(z_m-z_l)^2}\right]^k,\\
			\Pi_{N,k,i}^{(2)}(z^*_{i+1})&:=\ \ \,\prod_{1\leq l<i}\left[\frac{(z_i/\sqrt2-z_l-1/2)^2}{1/2+(z_i/\sqrt2-z_l-1/2)^2}\right]^{2k},\\
			\Pi_{N,k,i}^{(3)}(z^*_{i+1})&:=\prod_{i+1<m\leq N}\left[\frac{(z_m-z_i/\sqrt2-1/2)^2}{1/2+(z_m-z_i/\sqrt2 -1/2)^2}\right]^{2k}
		\end{align*}
		respectively. Therefore,
		\begin{align*}
			\mathbb{P}^x(X_1&\in E^r)\geq C_1(N,k)r^{2k+1}.
		\end{align*}
		For $r>1$ we start from the same point:
		\begin{align*}
			\mbb P^x(X_1\in E^r) &\geq K_1\int\limits_{E^r}\frac{e^{-\|x-y\|^2/2}w_k(y)}{\prod_{1\leq l<m\leq N}(1/2+(x_m-x_l)(y_m-y_l))^k}\intd y.
		\end{align*}
		If $x\in E^1$, then we use the same arguments as in the the case $r\leq 1$ by making use of
		\begin{align*}
			\mbb P^x(X_1\in E^r) &\geq\mbb P^x(X_1\in E^1).
		\end{align*}
		By this observation, we can use the calculations for the case $r\leq 1$ specialized to $r=1$ and the result is immediate. Otherwise, $x\in E^r\backslash E^1$, which means that after choosing
		\begin{align*}
			i=\argmin_{j=1,\ldots,N-1}(x_{j+1}-x_j),
		\end{align*}
		we have
		\begin{align*}
			x_{i+1}-x_i\geq 1.
		\end{align*}
		Consider the transformation $y=z+x$ with integration set $z\in B(0,1/2)\cap \{z:z_{i+1}-z_i\leq 0\}$. 
		We must ensure that any such $z$ implies that $y=x+z\in E_i^r$. 
		Because $1\leq x_{i+1}-x_i\leq r$ and $-1/\sqrt2\leq z_{j+1}-z_j$ for every $j\in\{1,\ldots,N-1\}$, we obtain
		\begin{align*}
			0<-\frac{1}{\sqrt2}+1\leq-\frac{1}{\sqrt2}+x_{i+1}-x_i\leq z_{i+1}-z_i+x_{i+1}-x_i=y_{i+1}-y_i\leq x_{i+1}-x_i\leq r.
		\end{align*}
		Here, we used the additional requirement that $z_{i+1}-z_i\leq 0$. For the rest of the pairs ($j\neq i$), we get a similar relationship, albeit without an upper bound,
		\begin{align*}
			0<-\frac{1}{\sqrt2}+1\leq-\frac{1}{\sqrt2}+x_{j+1}-x_j\leq z_{j+1}-z_j+x_{j+1}-x_j=y_{j+1}-y_j.
		\end{align*}
		Therefore, the $y$'s are still ordered and in $E_i^r$. 
		Consequently, we obtain
		\begin{align*}
			\mbb P^x(X_1\in E^r) &\geq K_1\int\limits_{B(0,1/2)\cap\{z:z_{i+1}-z_i\leq0\}}\frac{e^{-\|z\|^2/2}w_k(z+x)}{\prod_{1\leq l<m\leq N}(1/2+(x_m-x_l)(z_m-z_l+x_m-x_l))^k}\intd z\\
			&\geq K_1\int\limits_{B(0,1/2)\cap\{z:z_{i+1}-z_i\leq0\}}\frac{e^{-1/8}w_k(z+x)}{\prod_{1\leq l<m\leq N}(1/2+(x_m-x_l)(z_m-z_l+x_m-x_l))^k}\intd z.
		\end{align*}
		Let us look at all terms in the product:
		\begin{align*}
			\frac{(z_m-z_l+x_m-x_l)^2}{1/2+(z_m-z_l+x_m-x_l)(x_m-x_l)}&\geq
			\frac{(x_m-x_l-1/\sqrt2)^2}{1/2+(1/\sqrt{2}+x_m-x_l)(x_m-x_l)}\\
			&\geq
			\frac{(1-1/\sqrt2)^2}{1/2+(1/\sqrt{2}+1)\cdot 1}.
		\end{align*}
		In the first line, we used $-\frac{1}{\sqrt2}\leq z_m-z_l \leq \frac{1}{\sqrt2}$ and then the increasing nature of 
		$$\xi\mapsto \frac{(\xi-1/\sqrt2)^2}{1/2+(\sqrt2+\xi)\xi}$$
		for $\xi\geq 1$. Therefore, we can write
		\begin{align*}
			\mbb P^x(X_1\in E^r) &\geq K_1\left(\frac{(1-1/\sqrt2)^2}{1/2+1/\sqrt{2}+1}\right)^{kN(N-1)/2}e^{-1/8}\lambda^N\big(B(0,1/2)\cap\{z:z_{i+1}-z_i\leq0\}\big)\\
			&=K_1\left(\frac{(1-1/\sqrt2)^2}{1/2+1/\sqrt{2}+1}\right)^{kN(N-1)/2}\frac{\pi^{N/2}e^{-1/8}}{\Gamma(N/2+1)2^{N+1}}.  
		\end{align*}
		The Lebesgue-measure is independent of $i$ and positive which results in the desired lower bound. With this, we have proved the first part of the statement.
		
		For the second part, namely, the upper bound, we start by noting that if $r> 1$, we can write
		\begin{align*}
			P^x(X_1\in E^r)\leq 1
		\end{align*}
		so it suffices to consider the case $0<r\leq 1$. Now, we consider the bound
		\begin{align*}
			\mbb P^x(X_1\in E^r)&\leq K_2\int\limits_{E^r}\frac{e^{-\|x-y\|^2/2}w_k(y)}{\prod_{1\leq l<m\leq N}(1/2+(x_m-x_l)(y_m-y_l))^k}\intd y\\
			&\leq K_2\sum_{i=1}^{N-1}\int\limits_{E_i^r}\frac{e^{-\|x-y\|^2/2}w_k(y)}{\prod_{1\leq l<m\leq N}(1/2+(x_m-x_l)(y_m-y_l))^k}\intd y.
		\end{align*}
		Here, $E_i^r$ is as defined previously, while noting that
		\begin{align}
			E^r=\bigcup_{i=1}^{N-1}E_i^r=\bigcup_{i=1}^{N-1}\{y\in \overline{W}:0\leq y_{i+1}-y_i\leq r\}.\label{eq:UnionEir}
		\end{align}
		Looking at each integral, we can use the variable transformation given in \eqref{eq:maintransformation}. 
		Here, we only require that $z$ be such that $y$ covers at least $E_i^r$ while being in $\overline{W}$, and that $0\leq z_{i+1}\leq r/\sqrt2$. Then, we look at each of the parts of the probability density again, starting from the argument of the exponential function. We write 
		\begin{align*}
			\|x-y\|^2
			&=\|z\|^2+\frac{(x_{i+1}-x_i)^2}{2}+(x_{i+1}-x_i)\frac{z_i-z_{i+1}}{\sqrt{2}}-(x_{i+1}-x_i)\frac{z_i+z_{i+1}}{\sqrt{2}}\\
			&=\|z\|^2+\frac{(x_{i+1}-x_i)^2}{2}-\sqrt2 (x_{i+1}-x_i)z_{i+1}\\
			&\geq \sum_{\substack{j=1\\j\neq i+1}}^Nz_j^2+\frac{(x_{i+1}-x_i)^2}{2}-(x_{i+1}-x_i)\geq \sum_{\substack{j=1\\j\neq i+1}}^Nz_j^2-\frac{1}{2}
		\end{align*}
		since $0\leq z_{i+1}\leq r/\sqrt2\leq 1/\sqrt2$ and the minimum value of the function $\xi\mapsto\xi^2/2-\xi$ is $-1/2$. Note that the first line in this calculation is the same as the one we performed when we used the transformation to $z$ the first time. The rest of the terms come from the product 
		\begin{align*}
			\prod_{1\leq l<m\leq N}\left(\frac{(y_m-y_l)^2}{1/2+(x_m-x_l)(y_m-y_l)}\right)^k\leq 2^{kN(N-1)/2}\prod_{1\leq l<m\leq N}(y_m-y_l)^{2k}.
		\end{align*}
		We have taken away the term $(x_m-x_l)(y_m-y_l)\geq 0$ in the denominator. The simplest term is $l=i$, $m=i+1$,
		\begin{align*}
			(y_{i+1}-y_i)^2=2z_{i+1}^2.
		\end{align*}
		In the following, we often use the bound 
		\begin{align}
			\label{eq:cond_x_in_the_ball}
			0\leq x_m-x_l\leq \sqrt2 R
		\end{align}
		for every $1\leq l<m\leq N$ and $0\leq z_{i+1}/\sqrt2\leq r/2\leq 1/2$.
		We simply bound the terms $l,m\neq i,i+1$ 
		\begin{align*}
			y_m-y_l=z_m-z_l+x_m-x_l\leq z_m-z_l+\sqrt2 R.
		\end{align*}
		For $l=i$ and $m>i+1$, we obtain
		\begin{align*}
			0\leq y_m-y_i&=z_m+x_m-\frac{z_i-z_{i+1}}{\sqrt2}-\frac{x_{i+1}+x_i}{2}\\
			&=z_m-\frac{z_i}{\sqrt2}+\frac{z_{i+1}}{\sqrt2}+\frac{x_{m}-x_i}{2}+\frac{x_m-x_{i+1}}{2}\\
			&\leq z_m-\frac{z_i}{\sqrt2}+\frac12+\sqrt2 R.
		\end{align*}
		If $l=i+1$ and $m>i+1$, then
		\begin{align*}
			0\leq y_m-y_{i+1}&=z_m+x_m-\frac{z_i+z_{i+1}}{\sqrt2}-\frac{x_{i+1}+x_i}{2}\\
			&=z_m-\frac{z_i}{\sqrt2}-\frac{z_{i+1}}{\sqrt2}+\frac{x_{m}-x_i}{2}+\frac{x_m-x_{i+1}}{2}\\
			&\leq z_m-\frac{z_i}{\sqrt2}+\sqrt2R.
		\end{align*}
		When $m=i+1$ and $l<i$, we obtain
		\begin{align*}
			0\leq y_{i+1}-y_l&=\frac{z_i+z_{i+1}}{\sqrt2}+\frac{x_{i+1}+x_i}{2}-z_l-x_l\\
			&=\frac{z_i}{\sqrt2}-z_l+\frac{z_{i+1}}{\sqrt2}+\frac{x_i-x_l}{2}+\frac{x_{i+1}-x_l}{2}\\
			&\leq \frac{z_i}{\sqrt2}-z_l+\frac12+\sqrt2 R.
		\end{align*}
		The last case we must examine is that where $m=i$ and $l<i$,
		\begin{align*}
			0\leq y_{i}-y_l&=\frac{z_i-z_{i+1}}{\sqrt2}+\frac{x_{i+1}+x_i}{2}-z_l-x_l\\
			&=\frac{z_i}{\sqrt2}-z_l-\frac{z_{i+1}}{\sqrt2}+\frac{x_i-x_l}{2}+\frac{x_{i+1}-x_l}{2}\\
			&\leq \frac{z_i}{\sqrt2}-z_l+\sqrt2 R.
		\end{align*}
		Let us recall the integration space. Because $y\in\overline{W}$, the components of $z$ must satisfy
		\begin{align*}
			z_{j}+x_{j}&\leq  z_{j+1}+x_{j+1},\ \qquad j\not\in\{i-1,i,i+1\} \\
			z_{i-1}+x_{i-1}&\leq \frac{z_i-z_{i+1}}{\sqrt2}+\frac{x_{i+1}+x_i}{2},\\
			\frac{z_i+z_{i+1}}{\sqrt2}+\frac{x_{i+1}+x_i}{2}&\leq  z_{i+2}+x_{i+2},\\
			0&\leq \sqrt2 z_{i+1}\leq r.
		\end{align*}
		We can maximize the integration space by using the fact that $z_{i+1}\geq 0$
		\begin{align*}
			z_{j}+x_{j}&\leq  z_{j+1}+x_{j+1},\ \qquad j\not\in\{i-1,i,i+1\} \\
			z_{i-1}+x_{i-1}&\leq \frac{z_i}{\sqrt2}+\frac{x_{i+1}+x_i}{2},\\
			\frac{z_i}{\sqrt2}+\frac{x_{i+1}+x_i}{2}&\leq  z_{i+2}+x_{i+2},\\
			0&\leq \sqrt2 z_{i+1}\leq r
		\end{align*}
		and then using \eqref{eq:cond_x_in_the_ball}.
		Using these inequalities, we define the set
		\begin{align*}
			\mathcal{M}_{i,R,N}:=\bigg\{z\in\mathbb{R}^N:\ & z_j\leq z_{j+1}+\sqrt2 R \text{ for } j\not\in\{i-1,i,i+1\},z_{i-1}-\sqrt2R\leq \frac{z_i}{\sqrt2}\leq z_{i+2}+\sqrt2R \bigg\}
		\end{align*}
		in order to obtain the integration set
		\begin{align*}
			\mathcal{M}_{i,R,N}\cap \{0\leq z_{i+1}\leq r/\sqrt2\}.
		\end{align*}
		With these definitions, inequalities, and the projection \eqref{eq:projection}, we obtain
		\begin{align*}
			\mathbb{P}^x\big(X_1&\in E^r\big)\leq K_2\sum_{i=1}^{N-1}\int\limits_{E_{i}^{r}}\frac{e^{-\frac{\|x-y\|^2}{2}}w_k(y)}{\prod\limits_{1\leq l<m\leq N}\big(\frac12+(x_m-x_l)(y_m-y_l)\big)^k}\intd y\\
			&\leq K_2\sum_{i=1}^{N-1}\int\limits_{E_{i}^{r}}2^{\frac{kN(N-1)}{2}}e^{-\frac{\|x-y\|^2}{2}}w_k(y)\intd y\\
			&\leq K_2\sum_{i=1}^{N-1}\int\limits_{p_{i+1}(\mathcal{M}_{i,R,N})}\int\limits_0^{ \frac{r}{\sqrt2}}2^{\frac{kN(N-1)}{2}}e^{-\frac{\|z_{i+1}^\ast\|^2}{2}+\frac12}\\
			&\hspace*{3.5cm}2^kz_{i+1}^{2k}\Pi'^{(1)}_{N,k,R,i}(z_{i+1}^\ast)\Pi'^{(2)}_{N,k,R,i}(z_{i+1}^\ast)\Pi'^{(3)}_{N,k,i,R}(z_{i+1}^\ast)\intd z_{i+1}\intd z_{i+1}^\ast\\
			&=K_2\frac{2^{\frac{kN(N-1)-1}{2}}r^{2k+1}}{2k+1}\sum\limits_{i=1}^{N-1}\int\limits_{p_{i+1}(\mathcal{M}_{i,R,N})}e^{-\frac{\|z_{i+1}^\ast\|^2}{2}+\frac12}\\
			&\hspace*{4.9cm}\Pi'^{(1)}_{N,k,R,i}(z_{i+1}^\ast)\Pi'^{(2)}_{N,k,R,i}(z_{i+1}^\ast,x)\Pi'^{(3)}_{N,k,R,i}(z_{i+1}^\ast)\intd z_{i+1}^\ast\\
			&=:C_2(N,k,R)r^{2k+1}
		\end{align*}
		with
		\begin{align*}
			\Pi'^{(1)}_{N,k,R,i}(z_{i+1}^\ast)&:=\prod_{\substack{1\leq l<m\leq N\\l,m\neq i,i+1}}(z_m-z_l+\sqrt2 R)^{2k},\\
			\Pi'^{(2)}_{N,k,R,i}(z_{i+1}^\ast)&:=\prod_{i+1<m\leq N}\left[\left(z_m-\frac{z_i}{\sqrt2}+\frac{1}{2}+\sqrt2 R\right)\left(z_m-\frac{z_i}{\sqrt2}+\sqrt2 R\right)\right]^{2k},\\
			\Pi'^{(3)}_{N,k,R,i}(z_{i+1}^\ast)&:=\prod_{1\leq l<i}\left[\left(\frac{z_i}{\sqrt2}-z_l+\frac{1}{2}+\sqrt2 R\right)\left(\frac{z_i}{\sqrt2}-z_l+\sqrt2 R\right)\right]^{2k},
		\end{align*}
		which are all quantities independent of $r$ and $z_{i+1}$. This completes the proof.
	\end{proof}
	The previous lemma allows us to find the following bound.
	\begin{lemma}\label{lm:mainbound1}
		Given $ x\in\partial W$, then for every $\varepsilon>0$ there exists a constant $C_3(N,x,k,\varepsilon)>0$ such that
		\begin{align*}
			\mbb{P}^x[\exists s\in[t_1,t_2]:X_s\in E^r]\leq C_3 (t_2-t_1)^{k+1/2}
		\end{align*}
		for every $t_2>t_1\geq \varepsilon$  and $0<r<\sqrt{ t_2-t_1}$.
	\end{lemma}
	\begin{proof}[Proof of Lemma~\ref{lm:mainbound1}]
		We begin by making a quick derivation of the bound
		\begin{align}
			\mbb{P}^x[\exists s\in[t_1,t_2]: X_s\in E^r]\leq\frac{\int\limits_{t_1}^{2t_2-t_1}\mbb{P}^x[X_s\in E^r]\ud s}{\inf\limits_{ y\in  E^r} \int\limits_0^{t_2-t_1}\mbb{P}^y[X_u\in E^r]\ud u  },\label{eq:auxbound1}
		\end{align}
		which was proved in \cite[Theorem 4.1]{LiuXiao98} for the case where the process reaches the ball of radius $r$ centered at the origin. Here, we consider the more complicated set $E^r$ instead. Define the stopping time
		\begin{align*}
			T:=\inf\{s\geq t_1:X_s\in E^r\},
		\end{align*}
		understanding that if $X_t(\omega)$ does not hit $E^r$ in finite time, then $T(\omega)=\infty$. By definition
		\begin{align*}
			\mbb{P}^x[T\leq t_2]=\mbb{P}^x[\exists s\in[t_1,t_2]: X_s\in E^r].
		\end{align*}
		Now, the numerator on \eqref{eq:auxbound1} has a lower bound given by
		\begin{align*}
			\int\limits_{t_1}^{2t_2-t_1}\mbb{P}^x[X_s\in E^r]\ud s&=\mbb{E}^x\Big\{\int\limits_{t_1}^{2t_2-t_1}\mb{1}[X_s\in E^r]\ud s\Big\}\geq \mbb{E}^x\Big\{\int\limits_{T}^{2t_2-t_1}\mb{1}[X_s\in E^r]\Big\}\ud s\\
			&\geq\mbb{E}^x\Big\{\mb{1}[T\leq t_2]\int\limits_{T}^{2t_2-t_1}\mb{1}[X_s\in E^r]\ud s\Big\}\\
			&=\mbb{E}^x\Big\{\mb{1}[T\leq t_2]\ \mbb{E}^{X_T}\Big[\int\limits_{0}^{2t_2-t_1-T}\mb{1}[ X_u\in E^r]\ud u\Big]\Big\}\\
			&\geq \mbb{E}^x\Big\{\mb{1}[T\leq t_2]\ \mbb{E}^{X_T}\Big[\int\limits_{0}^{t_2-t_1}\mb{1}[X_u\in E^r]\ud u\Big]\Big\}.
		\end{align*}
		The first inequality results from $T \geq t_1$, the strong Markov property and the substitution $u=s-T$ yield the third line, and $T\leq t_2$ provides the last line. Now, we find a lower bound for the expectation started at $X_T$ using Fubini's theorem,
		\begin{align*}
			\mbb{E}^{X_T}\Big[\int\limits_{0}^{t_2-t_1}\mb{1}[X_u\in E^r]\ud u\Big]=\int\limits_{0}^{t_2-t_1}\mbb{P}^{X_T}[X_u\in E^r]\ud u\geq \inf_{y\in E^r}\int\limits_{0}^{t_2-t_1}\mbb{P}^{y}[X_u\in E^r]\ud u.
		\end{align*}
		This inequality is due to $X_T\in E^r$, which follows by the definition of $T$. Therefore,
		\begin{align*}
			\int\limits_{t_1}^{2t_2-t_1}\mbb{P}^x[X_s\in E^r]\ud s\geq \mbb{P}^x[T\leq t_2] \inf\limits_{y\in  E^r} \int\limits_{0}^{t_2-t_1}\mbb{P}^{y}[X_u\in E^r]\ud u,
		\end{align*}
		proving \eqref{eq:auxbound1}. 
		We proceed to use the $1/2$-semi-stability of $X_t$ and Lemma~\ref{lm:boundsA} to find that
		\begin{align*}
			\mbb{P}^x[T\leq t_2]\leq \frac{\int\limits_{t_1}^{2t_2-t_1}\mbb{P}^{x/\sqrt{s}}[X_1\in E^{r/\sqrt{s}}]\intd s}{\inf\limits_{y\in  E^r}\int\limits_0^{t_2-t_1}\mbb{P}^{y/\sqrt{u}}[X_1\in E^{r/\sqrt{u}}]\intd u}\leq\frac{C_2\int\limits_{t_1}^{2t_2-t_1}s^{-k-1/2}\intd s}{C_1\int\limits_0^{t_2-t_1}(t_2-t_1)^{-k-1/2}\intd u}.
		\end{align*}
		Here, the condition $r<\sqrt{t_2-t_1}$ ensures $\min(1,r/\sqrt{u})\geq \min(1,r/\sqrt{t_2-t_1})=r/\sqrt{t_2-t_1}$ in the denominator to obtain the second inequality. For the numerator we choose $R:=\|x\|/\sqrt \varepsilon$ so that \Cref{lm:boundsA} holds and $C_2$ is independent of $s$ which ensures we can take the constant out of the integral.  
		
		
		The numerator is bounded above by
		\begin{align*}
			\frac{C_2}{1/2-k}[(2t_2-t_1)^{1/2-k}-t_1^{1/2-k}]&=\frac{C_2t_1^{1/2-k}}{1/2-k}\Big[\Big(1+\frac{2(t_2-t_1)}{t_1}\Big)^{1/2-k}-1\Big]\\
			&\leq 2C_2\frac{t_2-t_1}{t_1^{k+1/2}}\leq \frac{2C_2}{\varepsilon^{k+1/2}}(t_2-t_1).
		\end{align*}
		Here, we have used the fact that $0<k<1/2$ and Bernoulli's inequality, as well as $\varepsilon\leq t_1$ to obtain the second line. Next, the denominator is
		\begin{align*}
			C_1 (t_2-t_1)^{1/2-k}
		\end{align*}
		so we can finally obtain
		\begin{align*}
			\mbb{P}^x[T\leq t_2]\leq C_3 (t_2-t_1)^{k+1/2},
		\end{align*}
		with $C_3=2C_2/C_1\varepsilon^{k+1/2}$.
	\end{proof}
	The previous lemma gives us the main tool required to prove the theorem. The following proof is largely adapted from \cite{LiuXiao98}.
	\begin{proof}[Proof of Theorem~\ref{theo_multidimensional_upper_bound}]
		We first restrict ourselves to prove that for every interval $[t_1,t_2]\subset (0,\infty)$ 
		\begin{align*}
			\dim \left( X^{-1}(E^0)\cap [t_1,t_2]\right)\leq \frac{1}{2}- k, 
		\end{align*}
		and extend this result to $(0,\infty)$ at the end. For $n\in \mathbb{N}$, we divide $[t_1,t_2]$ into $n$ subintervals of length $(t_2-t_1)/n$,
		\begin{align*}
			I_{n,i}:=[(i-1)(t_2-t_1)/n,i(t_2-t_1)/n].
		\end{align*}
		The $I_{n,i}$'s constitute a covering for $X^{-1}(E^0)\cap [t_1,t_2]$. We conclude
		\begin{align*}
			\mbb{E}^x\Big(m_{\frac{1}{2}-k}(X^{-1}(E^0)\cap [t_1, t_2])\Big)
			&\leq \lim\limits_{n\to \infty} \sum_{i=1}^n \Big(\frac{t_2-t_1}{n}\Big)^{\frac{1}{2}-k} \mbb{P}^x\Big( \exists t\in I_{n,i}:X_t \in E^0\Big)\\
			&\leq \lim\limits_{n\to \infty} \sum_{i=1}^n \Big(\frac{t_2-t_1}{n}\Big)^{\frac{1}{2}-k} \mbb{P}^x\Big(\exists t\in I_{n,i}:X_t \in E^{r_n}\Big)\\
			&\stackrel{\text{\Cref{lm:mainbound1}}}{\leq  } C_3 \lim\limits_{n\to \infty} \Big(\frac{t_2-t_1}{n}\Big)^{\frac{1}{2}-k}\Big(\frac{t_2-t_1}{n}\Big)^{k+\frac{1}{2}} n\\
			&=C_3(t_2-t_1)<\infty
		\end{align*}
		by choosing  $0<r_n<\sqrt{\frac{t_2-t_1}{n}}$. The first inequality holds since we are considering one particular covering instead of the optimal cover. Therefore, we deduce 
		\begin{align}
			\label{eq:upper_bound_with_stopping_time}
			\dim \left(X^{-1}(E^0)\cap [t_1,t_2]\right)\leq \frac{1}{2}-k	
		\end{align}
		$\mbb{P}^x$-almost surely as  $m_{\frac{1}{2}-k}(E^0)\in \{0,\infty\}$ by \Cref{lm_Hausdorff_dimension}. Finally, the result follows from the countable stability of the Hausdorff dimension and $\dim(\{0\})=0$,  
		\begin{align*}
			\dim \left(X^{-1}(E^0)\right)& =\dim\Big(\{0\}\cup \bigcup\limits_{n=1}^\infty X^{-1}(E^0)\cap \Big[\frac{1}{n},n\Big]\Big)\\
			&=\sup\limits_{n\in \mathbb{N}} \dim\left(X^{-1}(E^0)\cap \Big[\frac{1}{n},n\Big]\right)\leq \frac{1}{2}-k
		\end{align*}
		$\mbb{P}^x$-almost surely.
	\end{proof}
	\section{Lower Bound}\label{sc:lower}
	The objective of this section is to prove the following statement, which completes the proof of Theorem~\ref{th:main}.
	\begin{theorem}
		\label{theo_multidimensional_Hausdorff_dimension}
		For every $0<k<\frac{1}{2}$ and $x\in \partial W$ the inequality
		\begin{align*}
			\dim \left(X^{-1}(\partial W)\right)\geq \frac{1}{2}-k
		\end{align*}
		is $\mbb{P}^x$-almost surely valid.
	\end{theorem}
	In \cite{LiuXiao98}, the upper bound in their version of \Cref{lm:boundsA} has a constant $C_2$ that is independent of the starting point $x$ and its norm. In our case, this constant is a function of $R\geq \|x\|$, and therefore we must prove the following bound.
	\begin{lemma}
		\label{first_lemma_lower_bound}
		For $0<\varepsilon\leq s_1<s_2$ and $n\in \mathbb{N}$
		, we have
		\begin{align*}
			\mbb{P}^x&(X_{s_1}\in E^{\frac{1}{n}},X_{s_2}\in E^{\frac{1}{n}})\leq  c_k^{-2} \left(\frac{s_2-s_1}{s_2}\right)^{\frac{N}{2}+\gamma} \notag\\
			&\qquad \qquad \qquad \int\limits_{E^{\sqrt{\frac{s_2}{n^2s_1(s_2-s_1)}}}} \int\limits_{E^{\frac{1}{n\sqrt{s_2-s_1}}}} e^{-\frac{\|u\|^2}{2}-\frac{\|v\|^2}{2} +\|u\|\left(\frac{\|x\|}{\sqrt{\varepsilon}}  +\|v\|\right)} w_k(v )  w_k( u) \intd v \intd u,
		\end{align*}
		with $\gamma$ given in \eqref{eq:gamma}.
	\end{lemma}
	\begin{proof}
		We start from 
		\begin{align*}
			\mbb{P}^x(X_{s_1}\in E^{\frac{1}{n}},X_{s_2}\in E^{\frac{1}{n}} )&= \int\limits_{E^{\frac{1}{n}}} P_{k}(s_2-s_1,E^{\frac{1}{n}}\cond{} y) P_k(s_1, \ud y \cond{} x)
		\end{align*}
		with $x\in E^0=\partial W$ and $s_1 < s_2$. We first calculate 
		\begin{align*}
			\mbb{P}^x(X_{s_1}\in E^{\frac{1}{n}}&,X_{s_2}\in E^{\frac{1}{n}} )\\
			&= c_k^{-2} \int\limits_{E^{\frac{1}{n}}} \int\limits_{E^{\frac{1}{n}}} \frac{e^{-\frac{\|y\|^2+\|z\|^2}{2(s_2-s_1)}}}{ (s_2-s_1)^{\frac{N}{2}}} D_k\left(\frac{y}{\sqrt{s_2-s_1}}, \frac{z}{\sqrt{s_2-s_1}}\right) w_k\left(\frac{z}{\sqrt{s_2-s_1}}\right) \ud z \\
			&\qquad \qquad \qquad \qquad \qquad \qquad \qquad \cdot \frac{e^{-\frac{\|x\|^2+\|y\|^2}{2s_1}}}{s_1^{\frac{N}{2}}}D_k\left(\frac{x}{\sqrt{s_1}}, \frac{y}{\sqrt{s_1}}\right) w_k\left(\frac{y}{\sqrt{s_1}}\right) \ud y\\
			&= c_k^{-2} \int\limits_{E^{\frac{1}{n}}} \int\limits_{E^{\frac{1}{n\sqrt{s_2-s_1}}}} e^{-\frac{\|y\|^2}{2(s_2-s_1)}-\frac{\|v\|^2}{2}} D_k\left(\frac{y}{\sqrt{s_2-s_1}}, v \right) w_k\left(v\right) \ud v \\
			&\qquad \qquad \qquad \qquad \qquad \qquad \qquad \cdot \frac{e^{-\frac{\|x\|^2+\|y\|^2}{2s_1}}}{s_1^{\frac{N}{2}}}D_k\left(\frac{x}{\sqrt{s_1}}, \frac{y}{\sqrt{s_1}}\right) w_k\left(\frac{y}{\sqrt{s_1}}\right) \ud y.
		\end{align*}
		Here, we performed the substitution $v=z/\sqrt{s_2-s_1}$. Using $D_k(x,y)\leq e^{\|x\|\|y\|}$ yields 
		\begin{align*}
			\mbb{P}^x(X_{s_1}\in E^{\frac{1}{n}}&,X_{s_2}\in E^{\frac{1}{n}} )\leq \frac{1}{ c_k^2s_1^{\frac{N}{2}}}  \int\limits_{E^{\frac{1}{n}}} \int\limits_{E^{\frac{1}{n\sqrt{s_2-s_1}}}} e^{-\frac{\|y\|^2}{2(s_2-s_1)}-\frac{\|v\|^2}{2}+\frac{\|v\|\|y\|}{\sqrt{s_2-s_1}}}  e^{-\frac{\|x\|^2+\|y\|^2}{2s_1}+\frac{\|x\|\|y\|}{s_1}}  \\
			&\qquad\qquad \qquad\qquad\qquad\qquad \qquad\qquad \qquad\qquad\qquad\qquad \cdot w_k\left(v\right)  w_k\left(\frac{y}{\sqrt{s_1}}\right) \intd v \intd y\\
			&\leq  \frac{1}{ c_k^2s_1^{\frac{N}{2}}}  \int\limits_{E^{\frac{1}{n}}} \int\limits_{E^{\frac{1}{n\sqrt{s_2-s_1}}}} e^{-\frac{\|y\|^2}{2}\frac{s_2}{s_1(s_2-s_1)}-\frac{\|v\|^2}{2} +\|y\|\left(\frac{\|x\|}{s_1} +\frac{\|v\|}{\sqrt{s_2-s_1}}\right)} \\
			&\qquad\qquad\qquad\qquad \qquad\qquad \qquad\qquad \qquad\qquad\qquad\qquad \cdot w_k\left(v\right)  w_k\left(\frac{y}{\sqrt{s_1}}\right) \intd v \intd y.
		\end{align*}
		We have used the fact that $\|x\|^2/s_1\geq0$ to obtain this form. Next, we substitute $u=\sqrt{\frac{s_2}{s_1(s_2-s_1)}}y$ to obtain
		\begin{align*}
			\mbb{P}^x(X_{s_1}\in E^{\frac{1}{n}}&, X_{s_2}\in E^{\frac{1}{n}} )\leq  c_k^{-2}\left(\frac{s_2-s_1}{s_2}\right)^{\frac{N}{2}}\\
			&\qquad \int\limits_{E^{\sqrt{\frac{s_2}{n^2s_1(s_2-s_1)}}}} \int\limits_{E^{\frac{1}{n\sqrt{s_2-s_1}}}} e^{-\frac{\|u\|^2}{2}-\frac{\|v\|^2}{2} +\sqrt{\frac{s_1(s_2-s_1)}{s_2}}\|u\|\left(\frac{\|x\|}{s_1} +\frac{\|v\|}{\sqrt{s_2-s_1}}\right)}  \\
			&\qquad\qquad\qquad \qquad \qquad \qquad\qquad \qquad\qquad  \qquad\quad\cdot w_k(v)  w_k \left(\sqrt{\frac{s_2-s_1}{s_2} }u\right) \intd v \intd u\\
			&= c_k^{-2}\left(\frac{s_2-s_1}{s_2}\right)^{\frac{N}{2}+\gamma}\\
			& \qquad \int\limits_{E^{\sqrt{\frac{s_2}{n^2s_1(s_2-s_1)}}}} \int\limits_{E^{\frac{1}{n\sqrt{s_2-s_1}}}} e^{-\frac{\|u\|^2}{2}-\frac{\|v\|^2}{2} +\|u\|\left(\sqrt{\frac{s_2-s_1}{s_1s_2}} \|x\| +\sqrt{\frac{s_1}{s_2}}\|v\|\right)} \\
			&\qquad\qquad\qquad \qquad\qquad \qquad\qquad \qquad\qquad \qquad\qquad\qquad \qquad \ \cdot  w_k(v )  w_k( u) \intd v \intd u\\
			&\leq c_k^{-2} \left(\frac{s_2-s_1}{s_2}\right)^{\frac{N}{2}+\gamma} \\
			&\qquad \int\limits_{E^{\sqrt{\frac{s_2}{n^2s_1(s_2-s_1)}}}} \int\limits_{E^{\frac{1}{n\sqrt{s_2-s_1}}}} e^{-\frac{\|u\|^2}{2}-\frac{\|v\|^2}{2} +\|u\|\left(\frac{\|x\|}{\sqrt{\varepsilon}}  +\|v\|\right)} w_k(v )  w_k( u) \intd v \intd u.
		\end{align*}
		The equality holds due to $w_k(cy)=c^{2\gamma}w_k(y)$ for constants $c>0$ 
		and in the last step, we used $\varepsilon \leq s_1 < s_2$.
		\end{proof}
		This result holds without using any specific properties of the root system $A_{N-1}$, and therefore it holds for multivariate Bessel processes of any type. It only suffices to specialize the sum of multiplicities $\gamma$ and to define a similar set to $E^r$ for the corresponding root system, with the properties $E^{cr}=cE^r$, $c>0$ and $E^0=\partial W\subset E^r\subset \overline{W}$. Now, we will need one more bound. 
		\begin{lemma}\label{lm:doubleintegralA}
			For $0<\varepsilon\leq s_1< s_2$ and $n\in\mathbb N$, there exists $C(x,\varepsilon,N,k)$ such that
			\begin{align*}
				\mbb{P}^x(X_{s_1}\in E^{\frac{1}{n}},X_{s_2}\in E^{\frac{1}{n}} )&\leq\frac{C(x, \varepsilon, N,k)}{n^{4k+2}(s_2-s_1)^{k+\frac{1}{2}} s_1^{k+\frac{1}{2}}}.
			\end{align*}
		\end{lemma}
		\begin{proof}
			First, we analyze the integral in the statement of \Cref{first_lemma_lower_bound}
			\begin{align}
				\int\limits_{E^{\sqrt{\frac{s_2}{n^2s_1(s_2-s_1)}}}} \int\limits_{E^{\frac{1}{n\sqrt{s_2-s_1}}}} e^{-\frac{\|u\|^2}{2}-\frac{\|v\|^2}{2} +\|u\|\left(\frac{\|x\|}{\sqrt{\varepsilon}}  +\|v\|\right)} w_k(v )  w_k( u) \intd v \intd u.\notag
			\end{align}
			By \eqref{eq:UnionEir}, we see that this integral is bounded above by the sum over $i\in\{1,\ldots,N-1\}$ of the integral over $u$ of
			\begin{align}
				&
				\int\limits_{E_i^{\frac{1}{n\sqrt{s_2-s_1}}}}e^{-\frac{\|u\|^2}{2}-\frac{\|v\|^2}{2} +\|u\|\left(\frac{\|x\|}{\sqrt{\varepsilon}}  +\|v\|\right)}w_k(v )  
				\intd v
				.\notag
			\end{align}
			For each $i$ we perform the rotation
			\begin{align}
				z_j=&v_j\quad\text{for }j\neq i,i+1,\notag\\
				z_{i+1}=&\frac{v_{i+1}-v_i}{\sqrt{2}},\notag\\
				z_i=&\frac{v_{i+1}+v_i}{\sqrt{2}}.\notag
			\end{align}
			Then, the weight function can be written as
			\begin{align*}
				w_k&(v)=\prod_{1\leq l<m\leq N}(v_m-v_l)^{2k}\nonumber\\
				&=2^kz_{i+1}^{2k}\prod_{\substack{1\leq l<m\leq N\\m,l\neq i,i+1}}(z_m-z_l)^{2k}\prod_{l<i}\Big(\frac{z_i-z_{i+1}}{\sqrt{2}}-z_l\Big)^{2k}\Big(\frac{z_i+z_{i+1}}{\sqrt{2}}-z_l\Big)^{2k}\nonumber\\
				&\quad\qquad\qquad\prod_{i+1<m}\Big(z_m-\frac{z_i-z_{i+1}}{\sqrt{2}}\Big)^{2k}\Big(z_m-\frac{z_i+z_{i+1}}{\sqrt{2}}\Big)^{2k}\nonumber\\
				&=2^kz_{i+1}^{2k}\prod_{\substack{1\leq l<m\leq N\\m,l\neq i,i+1}}(z_m-z_l)^{2k}\prod_{l<i}\Big(\Big(\frac{z_i}{\sqrt{2}}-z_l\Big)^2-\frac{z_{i+1}^2}{2}\Big)^{2k}\prod_{i+1<m}\Big(\Big(z_m-\frac{z_i}{\sqrt{2}}\Big)^2-\frac{z_{i+1}^2}{2}\Big)^{2k}\\
				&\leq 2^kz_{i+1}^{2k}\prod_{\substack{1\leq l<m\leq 	N\\m,l\neq i,i+1}}(z_m-z_l)^{2k}\prod_{l<i}\Big(\frac{z_i}{\sqrt{2}}-z_l\Big)^{4k}\prod_{i+1<m}\Big(z_m-\frac{z_i}{\sqrt{2}}\Big)^{4k}\\
				&=:2^k z_{i+1}^{2k} \widetilde{\Pi}_{N,k,i}(z_{i+1}^*).
			\end{align*}
			Here, $z_{i+1}^*$ is given by \eqref{eq:projection}.
			Now, we recall that because $v\in E_i^{\frac{1}{n\sqrt{s_2-s_1}}}$
			the following bounds must hold,
			\begin{align}
				z_j&\leq z_{j+1} \text{ for } j\neq i-1,i,i+1\notag\\
				0\leq z_{i+1}&\leq 1/n\sqrt{2(s_2-s_1)},\notag\\
				\sqrt{2}z_{i-1}+z_{i+1}&\leq z_i\leq \sqrt2 z_{i+2}-z_{i+1}.\notag
			\end{align}
			We use these inequalities to define the set
			\begin{align*}
				\Lambda_{i,N}:=\{z\in\mathbb{R}^N:z_j\leq z_{j+1} \text{ for } j\notin \{i-1,i,i+1\}, \sqrt{2}z_{i-1}+z_{i+1}\leq z_i\leq \sqrt2 z_{i+2}-z_{i+1}\}.
			\end{align*}
			Therefore, our new integration set after the transformation is $\Lambda_{i,N}\cap \{0\leq z_{i+1}\leq 1/n\sqrt{2(s_2-s_1)}\}.$
			We define a new $N$-dimensional vector by $(\psi(z_{i+1}^*))_{i+1}:=z_i-\sqrt2 z_{i-1}$ and $(\psi(z_{i+1}^*))_j=z_j$ for every $j\neq i+1$ so that $\|z_{i+1}^*\|\leq \|z\|\leq \|\psi(z_{i+1}^*)\|$ for every $z\in\Lambda_{i,N}\cap \{0\leq z_{i+1}\leq 1/n\sqrt{2(s_2-s_1)}\}$. 
			Because $z_{i+1}\geq 0$, we can write
			\begin{align*}
				\Lambda'_{i,N}:=\{z\in\mathbb{R}^N:z_j\leq z_{j+1} \text{ for } j\notin\{ i-1,i,i+1\}, z_{i-1}\leq z_i/\sqrt{2}\leq  z_{i+2}\}
			\end{align*}
			and we can enlarge the integration set to $\Lambda'_{i,N}\cap\{0\leq z_{i+1}\leq 1/n\sqrt{2(s_2-s_1)}\}$. Because the variable substitution $v\to z$ is a rotation, norms are preserved and with these bounds we find the upper bound 
			\begin{align*}
				& \int\limits_{E_i^{\frac{1}{n\sqrt{s_2-s_1}}}}e^{-\frac{\|u\|^2}{2}-\frac{\|v\|^2}{2} +\|u\|\left(\frac{\|x\|}{\sqrt{\varepsilon}}  +\|v\|\right)}w_k(v )    \intd v\\
				&\leq   \int\limits_{\Lambda_{i,N}\cap\{0\leq z_{i+1}\leq 1/n\sqrt{2(s_2-s_1)}\} }e^{-\frac{\|u\|^2}{2}-\frac{\|z\|^2}{2} +\|u\|\left(\frac{\|x\|}{\sqrt{\varepsilon}}  +\|z\|\right)}2^kz_{i+1}^{2k}   \widetilde{\Pi}_{N,k,i}(z_{i+1}^\ast) \intd z\\
				&\leq   \int\limits_{\Lambda_{i,N}\cap\{0\leq z_{i+1}\leq 1/n\sqrt{2(s_2-s_1)}\} }e^{-\frac{\|u\|^2}{2}-\frac{\|z_{i+1}^*\|^2}{2} +\|u\|\left(\frac{\|x\|}{\sqrt{\varepsilon}}  +\|\psi(z_{i+1}^*)\|\right)}2^kz_{i+1}^{2k}   \widetilde{\Pi}_{N,k,i}(z_{i+1}^\ast) \intd z\\
				&\leq  \int\limits_{\Lambda'_{i,N}\cap\{0\leq z_{i+1}\leq 1/n\sqrt{2(s_2-s_1)}\} }e^{-\frac{\|u\|^2}{2}-\frac{\|z_{i+1}^*\|^2}{2} +\|u\|\left(\frac{\|x\|}{\sqrt{\varepsilon}}  +\|\psi (z_{i+1}^*)\|\right)}2^kz_{i+1}^{2k}   \widetilde{\Pi}_{N,k,i}(z_{i+1}^\ast) \intd z\\
				&\leq \frac{1}{\sqrt2n^{2k+1}(s_2-s_1)^{k+1/2}}\int\limits_{p_{i+1}(\Lambda'_{i,N})} e^{-\frac{\|u\|^2}{2}-\frac{\|z_{i+1}^*\|^2}{2} +\|u\|\left(\frac{\|x\|}{\sqrt{\varepsilon}}  +\|\psi (z_{i+1}^*)\|\right)}   \widetilde{\Pi}_{N,k,i}(z_{i+1}^\ast) \intd z_{i+1}^*.
			\end{align*}
			The first inequality follows from the rotation $v\to z$ and from the upper bound on the weight function, the second inequality is obtained from the relationships between $z$, $z_{i+1}^*$ and $\psi (z_{i+1}^*)$, the third inequality is obtained by enlarging the integration set and the final inequality is obtained from the integral with respect to $z_{i+1}$. Note that in the second to last line, only the factor $z_{i+1}^{2k}$ depends on $z_{i+1}$, which yields the last line via an upper bound on the integral over $z_{i+1}$. 

			This last expression depends on the index $i=1,\ldots,N-1$, and the functions of $u$ in the integrand are of similar form to the functions of $v$ in the integral we used as our starting point.
			Therefore, we can use a similar bound and decomposition of the integral over $u$, this time using the index $j=1,\ldots,N-1$ and with the function $s_1(s_2-s_1)/s_2$ replacing $(s_2-s_1)$. Obviously, the integral that results from these operations is bounded due to the fact that the second order terms in the exponential have negative coefficients. Hence, we receive by similar calculations and the same type of rotation the bound
			\begin{align*}
				\frac{C_{ij}(x, \varepsilon, N,k)}{\sqrt2 n^{2k+1}(s_2-s_1)^{k+\frac{1}{2}}} \frac{s_2^{k+\frac{1}{2}}}{\sqrt2n^{2k+1} s_1^{k+\frac{1}{2}}(s_2-s_1)^{k+\frac{1}{2}}},
			\end{align*}
			where clearly $C_{ij}$ is a positive bounded constant. To obtain a bound on the total integral, it suffices to add over both $i$ and $j$, so we define
			\begin{align*}
				\widehat{C}(x, \varepsilon, N,k):=\sum_{i,j=1}^{N-1}C_{ij}(x, \varepsilon, N,k),
			\end{align*}
			and finally, we obtain
			\begin{align*}
				\mbb{P}^x(X_{s_1}\in E^{\frac{1}{n}},&X_{s_2}\in E^{\frac{1}{n}} )\\&\leq \widehat{C}(x, \varepsilon, N,k)\frac{1}{ 2c_k^2}\underbrace{\left(\frac{s_2-s_1}{s_2}\right)^{\frac{(kN+1)(N-1)}{2}-k}}_{\leq 1}\frac{1}{n^{4k+2}(s_2-s_1)^{k+\frac{1}{2}} s_1^{k+\frac{1}{2}}}\\
				&\leq \widehat{C}(x, \varepsilon, N,k)\frac{1}{2 c_k^2}\frac{1}{n^{4k+2} (s_2-s_1)^{k+\frac{1}{2}} s_1^{k+\frac{1}{2}}}\\
				&=:{C}(x, \varepsilon, N,k)\frac{1}{n^{4k+2}(s_2-s_1)^{k+\frac{1}{2}} s_1^{k+\frac{1}{2}}}
			\end{align*}
			with $\varepsilon\leq s_1  < s_2$. 
		\end{proof}
		We prove the theorem by making use of the capacity argument, which will require the bounds we derived in the previous lemmas.
		\begin{proof}[Proof of Theorem \ref{theo_multidimensional_Hausdorff_dimension}]
			For the verification of the lower bound we consider an interval $[\varepsilon,t]\subset (0,\infty)$ and verify 
			$$\dim \left(X^{-1}(\partial W)\cap [\varepsilon,t] \right)\geq \frac{1}{2}-k$$ 
			according to the same scheme as \cite[Theorem 4.1 and 4.2]{LiuXiao98}. We construct for every possible Hausdorff dimension $0<\beta <\frac{1}{2}-k$ a positive measure $\mu $ on $X^{-1}(\partial W)\cap [\varepsilon,t]$ such that 
			\begin{align}
				\label{eq_measure_beta}
				\|\mu \|_\beta := \int\limits_\varepsilon^t \int\limits_\varepsilon^t \frac{\mu(\ud s_1)\mu (\ud s_2)}{|s_2-s_1|^\beta}, 
			\end{align}
			which implies $\dim \left( X^{-1}(\partial W)\cap [\varepsilon,t]\right)>\beta$ on $\{\mu >0\}$ due to the capacity argument, \Cref{lem_capacity}.
			Let $\mathcal{M}_\beta^+$ be the set of all non-negative measures on $\mathbb{R}^+$ with \eqref{eq_measure_beta}. Due to \cite{Adler1981} $\mathcal{M}_\beta^+$ is a complete metric space under the metric $\|\cdot \|_\beta$.
			We define a sequence of random positive measures on $\mathcal{B}(\mathbb{R}^+)$ by 
			\begin{align*}
				\mu_n(B, \omega ):=n^{2k+1} \int\limits_{[\varepsilon,t]\cap B}  \mb{1}[X_s(\omega)\in E^\frac{1}{n}]  \ud s.
			\end{align*}
			It was shown by Testard \cite{Testard1987} that, if there exist constants $K_1, K_2>0$ such that 
			\begin{align}
				\label{eq_proof}
				\mbb{E}^x(\|\mu_n\|)\geq K_1, \quad \mbb{E}^x(\|\mu_n\|^2)\leq K_2,\quad \mbb{E}^x(\|\mu_n\|_\beta)<\infty,
			\end{align}
			where $\|\mu_n\|=\mu_n([\varepsilon,t],\cdot )$, then there exists a subsequence $(\mu_{n_k})_k$ such that $\mu_{n_k}\to \mu $ in $\mathcal{M}_\beta^+$ and $\mu $ is strictly positive with probability at least $\frac{K_1^2}{2K_2}$. 
			Via a private communication with Y. Xiao, we have found that Testard's lemma can be adapted to our setting with minimal modifications, namely, the change in intervals from $[0,t]$ to $[\varepsilon,t]$.
			As $(X_t)_{t\geq 0}$ has continuous sample paths a.s., we see that $\supp(\mu_n)\subset\supp(\mu_m)$ whenever $n>m $, and in particular $\supp(\mu)\subset\supp(\mu_n)$ for every $n$. By the construction of $\mu_n$, any open Borel set $B\subset [\varepsilon,t]$ for which $\mu(B)>0$ must include a collision time with the boundary, that is, an element of $X^{-1}(\partial W)\cap [\varepsilon,t]$. Therefore, we see that $X^{-1}(\partial W)\cap [\varepsilon,t]$ is the support of $\mu$ (see \cite[p. 282]{Marcus1976}).
			Thus, our task is to prove \eqref{eq_proof}. By using the $1/2$-semi stability of the multivariate Bessel process we obtain
			\begin{align*}
				\mbb{E}^x(\|\mu_n\|)&= n^{2k+1}\int\limits_\varepsilon^t \mbb{P}^x(X_s\in E^\frac{1}{n}) \intd s\\
				&=n^{2k+1}\int\limits_\varepsilon^t P_{k}(s,E^\frac{1}{n}\cond{}x) \intd s\\
				&=n^{2k+1}\int\limits_\varepsilon^t P_{k}\left(1,E^\frac{1}{n\sqrt{s}}\big|\frac{x}{\sqrt{s}}\right)\intd s\\
				&\overset{\ref{lm:boundsA}}{\underset{}{\geq}} C_1n^{2k+1} n^{-2k-1}\int\limits_{\varepsilon}^t s^{-k-\frac{1}{2}}\intd s \\
				&=\frac{C_1}{\frac{1}{2}-k}\left(t^{\frac{1}{2}-k}-\varepsilon^{\frac{1}{2}-k}\right)=: K_1 > 0.
			\end{align*} 
			For the last two relationships in \eqref{eq_proof}, we use the result from Lemma~\ref{lm:doubleintegralA}.
			\begin{align*}
				\mbb{E}^x(\|\mu_n\|^2)&=n^{4k+2}\int\limits_\varepsilon^t \int\limits_\varepsilon^t \mbb{P}^x\left(X_{s_1}\in E^\frac{1}{n}, X_{s_2}\in \mbb{E}^\frac{1}{n}\right) \intd s_2 \intd s_1\\
				&\leq 2n^{4k+2}\int\limits_\varepsilon^t \int\limits_{s_1}^t \frac{{C}}{n^{4k+2}(s_2-s_1)^{k+\frac{1}{2}} s_1^{k+\frac{1}{2}}} \intd s_2 \intd s_1\\
				&= 2{C}\int\limits_\varepsilon^t \int\limits_{s_1}^t (s_2-s_1)^{-k-\frac{1}{2}} s_1^{-k-\frac{1}{2}} \intd s_2 \intd s_1\\
				&= \frac{2{C}}{1/2-k}\int\limits_\varepsilon^t  (t-s_1)^{\frac{1}{2}-k} s_1^{-k-\frac{1}{2}} \intd s_1\\
				&\leq \frac{2{C}}{1/2-k}t^{1/2-k}\int\limits_0^t  s_1^{-k-\frac{1}{2}} \intd s_1=\frac{2{C}}{(1/2-k)^2}t^{1-2k}=:K_2.
			\end{align*}
			This object is clearly positive and finite for $\varepsilon>0$, $0<k<1/2$. Now, we turn to the $\beta$-capacity for $0<\beta <1/2-k$; again, we make use of Lemma~\ref{lm:doubleintegralA}.
			\begin{align*}
				\mbb{E}^x(\|\mu_n\|_\beta)&=\mbb{E}^x\left( \int\limits_\varepsilon^t \int\limits_\varepsilon^t \frac{\mu(\ud s_1)\mu (\ud s_2)}{|s_2-s_1|^\beta} \right)\\
				&=2n^{4k+2}\int\limits_\varepsilon^t \int\limits_{s_1}^t \frac{\mbb{P}^x(X_{s_1}\in E^\frac{1}{n},X_{s_2}\in E^\frac{1}{n})}{(s_2-s_1)^\beta} \intd s_2 \intd s_1 \\
				&\leq
				2{C}\int\limits_\varepsilon^t \int\limits_{s_1}^t (s_2-s_1)^{-k-\beta-\frac{1}{2}} s_1^{-k-\frac{1}{2}}  \intd s_2 \intd s_1 \\
				&=\frac{2{C}}{\frac{1}{2}-k-\beta}\int\limits_\varepsilon^t  (t-s_1)^{\frac{1}{2}-k-\beta} s_1^{-k-\frac{1}{2}}   \intd s_1\leq  \frac{2{C}t^{\frac{1}{2}-k-\beta}}{\frac{1}{2}-k-\beta}\int\limits_\varepsilon^t   s_1^{-k-\frac{1}{2}}   \intd s_1\\
				&= \frac{2{C}t^{\frac{1}{2}-k-\beta}}{(\frac{1}{2}-k-\beta)(\frac{1}{2}-k)} \left(t^{\frac{1}{2}-k}-\varepsilon^{\frac{1}{2}-k}\right)\leq\frac{2{C}t^{1-2k-\beta}}{(\frac{1}{2}-k-\beta)(\frac{1}{2}-k)} <\infty
			\end{align*}
			whenever $0\leq\beta<\frac12-k$.

			With this, we have proved that with $\mbb{P}^x$ probability at least $\frac{K_1^2}{2K_2}$ (depending on $k, N, \varepsilon, t, x$)
			\begin{align}
				\label{eq_lower_bound_hausdorff_dimension}
				\dim \left(X^{-1}(\partial W)\cap [\varepsilon,t] \right)\geq \frac{1}{2}-k
			\end{align} 
			for every $0<\varepsilon<t<\infty$.
			Next, we define a subset for the hitting times of the Weyl chamber as a sequence of well-defined stopping times $\tau_n$ as $\tau_0:=0$ and 
			\begin{align*}
				\tau_n:=\inf \{t\geq \tau_{n-1}+1: X_t\in \partial W\}.
			\end{align*}
			By the strong Markov property, $X_{\tau_n}\in \partial W$ almost surely, and by \eqref{eq_lower_bound_hausdorff_dimension} we determine
			\begin{align*}
				&\mbb{P}^x\left(\dim \left( X^{-1}(\partial W)\cap [0,\tau_{n+1}]\right)\geq \frac{1}{2}-k\cond{\Big } \mathcal{F}_{\tau_n}\right) \\
				\geq\ &\mbb{P}^x\left(\dim \left(X^{-1}(\partial W)\cap [\tau_n+\varepsilon,\tau_{n+1}]\right)\geq \frac{1}{2}-k\cond{\Big} \mathcal{F}_{\tau_n}\right)\\
				=\  &\mbb{P}^{X_{\tau_n}}\left(\dim \left(X^{-1}(\partial W)\cap [\varepsilon,\tau_{n+1}-\tau_n]\right)\geq \frac{1}{2}-k\right)\\
				\geq\  &\mbb{P}^{X_{\tau_n}}\left(\dim \left( X^{-1}(\partial W)\cap [\varepsilon,1]\right)\geq \frac{1}{2}-k\right)\geq \frac{K_1^2}{2K_2}
			\end{align*}
			for some $0<\varepsilon<1$.	The sequence $A_n:= \{ \dim \left( X^{-1}(\partial W)\cap [0,\tau_{n}]\right)\geq \frac{1}{2}-k\}\in \mathcal{F}_{\tau_n}$ fulfills 
			\begin{align*}
				\sum\limits_{n\in \mathbb{N}} \mbb{P}^x(A_n \cond{} \mathcal{F}_{\tau_{n-1}})\geq \sum\limits_{n\in\mathbb{N}} \frac{K_1^2}{2K_2} = \infty
			\end{align*}
			almost surely and hence by the conditional Borel-Cantelli lemma \cite[Corollary 13.3.38]{bremaud2020probability},
			\begin{align*}
				\left\{ \sum\limits_{n\in \mathbb{N}} \mbb{P}^x(A_n \cond{} \mathcal{F}_{\tau_{n-1}}) = \infty\right\}\equiv \left\{ \sum\limits_{n\in\mathbb{N}}\mb{1}[A_n] =\infty \right\},
			\end{align*}
			we conclude that
			\begin{align*}
				\mbb{P}^x\left(\dim \left(X^{-1}(\partial W)\cap [0,\tau_{n}]\right)\geq \frac{1}{2}-k  \textrm{ for infinitely many }n \right)=1.
			\end{align*} 
			It only remains to prove that $\sup\{t : X_t\in \partial W\}=\infty$ almost surely. We set $T:=\inf\{t>0: X_t\in \partial W\}$, which satisfies $\mbb{P}^x(T<\infty)=1$ \cite{Demni09}, and deduce
			\begin{align*}
				\mbb{P}^x(\tau_{n+1}<\infty)&= \mbb{E}^x(\mb{1}[\tau_{n+1}<\infty])\\
				&=\mbb{E}^x(\mbb{E} (\mb{1}[\tau_{n+1}<\infty]\cond{}  \mathcal{F}_{\tau_n +1}))\\
				&=\mbb{E}^x(\mbb{E} (\mb{1}[\inf \{t\geq 0: X_{t+\tau_{n}+1}\in \partial W\} <\infty]\cond{}  \mathcal{F}_{\tau_n +1}))\\
				&=\mbb{E}^x(\mbb{E} (\mb{1}[\inf \{t\geq 0: X_{t}\in \partial W\} <\infty]\cond{}  x_0=X_{\tau_n +1}))\\
				&=\mbb{E}^x(\mbb{E} (\mb{1}[T <\infty]\cond{}  x_0=X_{\tau_n +1}))\\
				&=\mbb{E}^x(\mb{1}[T <\infty])=\mbb{P}^x(T<\infty)=1.
			\end{align*}
			The fourth line is valid due to the strong Markov property. Consequently, 
			\begin{align*}
				\mbb{P}^x(\forall n \in \mbb{N}: \tau_n<\infty)&=\mbb{P}^x\left(\bigcap_{n\in \mbb{N}}\{\tau_n<\infty\}\right)\\
				&= \lim_{n\to \infty}\mbb{P}^x(\tau_n<\infty)=1.
			\end{align*}
			Let us consider $\omega\in \{\forall n \in \mbb{N}: \tau_n<\infty\}$ and $C>0$. Due to the definition of $\tau_n$, $\tau_n\geq n$. Hence, $C< \tau_{n}$ for every $n\geq [C]+1$ (here, $[C]$ is the floor function of $C$). This, combined with $\tau_{n}(\omega)<\infty$ yields 
			\begin{align*}
				\{\forall n\in \mbb{N}: \tau_n<\infty\}\subset \{\sup\{t:X_t\in \partial W\}=\infty\}
			\end{align*} 
			and 
			\begin{align*}
				\mbb{P}^x(\{\sup\{t:X_t\in \partial W\}=\infty\})\geq \mbb{P}^x(\{\forall n\in \mbb{N}: \tau_n<\infty\} )=1.
			\end{align*}
		\end{proof}
		
		\section{Concluding remarks}\label{sc:conclusions}
		
		We have shown that collision times for multivariate Bessel processes of type $A_{N-1}$ have a Hausdorff dimension of the same form as that of a Bessel process. This is one more aspect that multivariate Bessel processes share with their single-variable counterpart, but more importantly, this result characterizes the behavior of these particle systems over their parameter space. In particular, Theorem~\ref{th:main} clearly establishes the change in behavior that occurs at $k=1/2$, namely, the point where these systems become non-colliding as $k$ grows. From a physics perspective, if we were to regard the Hausdorff dimension as a thermodynamic function, our result points to a behavior similar to a solid-liquid transition as $k$ decreases toward the critical value $k=1/2$. More explicitly, for lower temperatures ($k\geq 1/2$) there are no collisions, meaning that particle motion is strongly restricted by the long-range interaction (solid-like behavior); in contrast, for higher temperatures ($0<k<1/2$) the repulsion is still present but it is weak enough to allow collisions (liquid-like behavior). Finally, at infinite temperature ($k=0$) the repulsion disappears completely and particles move freely up to collisions (gas-like behavior).
		
		From the proofs presented here, it is easily seen that the Hausdorff dimension of collisions between neighboring particles behaves similarly. It suffices to consider the set
		$\partial W^i:= \{y\in \overline{W} : y_{i+1}=y_i\}\subset \partial W,$ and notice that our proofs work equally well with it, so we conclude that
		\begin{align*}
			\dim\big(X^{-1}(\partial W^i)\big)=\frac12-\min\Big(\frac12,k\Big)
		\end{align*}
		almost surely.
		In the proofs we can directly see that we already calculated the corresponding Lemmas by just looking at the smaller edge set $E^r_i\subset E^r$.

		We have only proved the result for the root system of type $A_{N-1}$, but we expect a similar result to be valid for multivariate Bessel processes with small enough multiplicities of all other types, as they are $1/2$-semi-stable and are known to hit the boundaries of their Weyl chambers almost surely. We conjecture that, for a process $(X_t)_{t\geq0}$ with root system $R$ and multiplicity function $k:R\to (0,\infty)$, we get as Hausdorff dimension $$\dim(X^{-1}(\partial W))=\frac12-\min\Big(\frac12,\min\limits_{\alpha\in R} k(\alpha)\Big)$$ 
		almost surely. Because each multiplicity represents the strength of the repulsion from its corresponding borders of the Weyl chamber, if any multiplicity is less than $1/2$ we can expect the process to collide with those borders. In order to generalize this result, it will be necessary to prove analogues of Lemma~\ref{lm:boundsA}, which in turn require asymptotics for the Dunkl-Bessel function similar to those obtained by Graczyk and Sawyer for the case $A_{N-1}$, and \Cref{lm:doubleintegralA}.
		
		An important feature of this paper is that the particle number $N$ remains fixed throughout. We expect that Theorem~\ref{th:main} remains valid in the limit $N\to\infty$, as there are related results that suggest this to be the case. One such result is given in \cite{Andraus20}, where the dynamics of Dunkl jump processes are considered. The latter processes are the discontinuous part of the corresponding Dunkl process. For instance, in the $A_{N-1}$ case the discontinuous part is a permutation of the labels of particles in the Dyson model. For $k>1/2$, the jump process is a Poisson random walk over the symmetric group where each step is a pairwise permutation, and the probability of each step is given by the inverse of the squared distance between the particles involved in the label exchange. There, it was shown that the total jump rate per particle per unit time diverges as $k\downarrow 1/2$ in the limit where $N\to\infty$. Moreover, the jump process dynamics are undefined for $0<k<1/2$. This fact suggests that the Hausdorff dimension of collision times may have the same form even in the infinite-particle limit.
		
		Though we expect that the results given here can be extended to all types of Dunkl processes, the semi-stability requirement rules out several interesting multivariate processes. For example, the $\beta$-Jacobi process \cite{Demni10} and all other multiple-particle stochastic systems defined in a bounded subset of the real line are not 1/2-semi-stable and are therefore outside of the scope of the result.

		\section*{Acknowledgments} 
		
		The authors would like to thank N. Hatano, J. Nagel, and J.H.C. Woerner for insightful comments on this work, to  M. Voit for pointing out the results in \cite{GraczykSawyer}, and Y. Xiao for his comments on the application of \cite{Testard1987} in the proof of Thm.~\ref{theo_multidimensional_Hausdorff_dimension}. Finally, the authors would like to thank two anonymous referees for thoughtful comments which led to the improvement of this paper. N.H. is completely supported by DFG-GRK 2131, especially for the trip to Japan to prepare this paper. S.A. is supported by the JSPS Kakenhi Grant number JP19K14617.
		
		\bibliography{bibtex}
	\end{document}